\documentclass[12pt]{amsart}

\usepackage[leqno]{amsmath}
\usepackage{amsthm}
\usepackage{amsfonts}
\usepackage{amssymb}
\usepackage[new]{old-arrows}
\usepackage{stmaryrd}
\usepackage{hyperref}
\usepackage{eucal}
\usepackage{mathtools}
\usepackage{tikz}
\usetikzlibrary{arrows,automata,decorations.markings,intersections, shapes.geometric,matrix,positioning,patterns}
\tikzset{
    state/.style={
           rectangle,
           rounded corners,
           draw=black, very thick,
           minimum height=2em,
           inner sep=2pt,
           text centered,
           },
}
\usepackage{tikz-cd}

\DeclareMathOperator{\SL}{SL}
\DeclareMathOperator{\GL}{GL}

\DeclareMathOperator{\Spec}{Spec}

\DeclareMathOperator{\Proj}{Proj}

\DeclareMathOperator{\Supp}{Supp}

\newcommand*{\vv}{\ensuremath{\mathbf{v}}}      
\newcommand*{\uu}{\ensuremath{\mathbf{u}}}  
\newcommand*{\ww}{\ensuremath{\mathbf{w}}}   
\newcommand*{\dd}{\ensuremath{\mathbf{d}}} 
\newcommand*{\ee}{\ensuremath{\mathbf{e}}}    
\newcommand*{\pp}{\ensuremath{\mathbf{p}}} 

\newcommand*{\g}{\ensuremath{\mathfrak{g}}} 

\newtheorem{innercustomthm}{Theorem}
\newenvironment{customthm}[1]
  {\renewcommand\theinnercustomthm{#1}\innercustomthm}
  {\endinnercustomthm}

\theoremstyle{plain}
  \newtheorem{theorem}[subsection]{Theorem}
  \newtheorem{proposition}[subsection]{Proposition}
  \newtheorem{lemma}[subsection]{Lemma}
  \newtheorem{corollary}[subsection]{Corollary}

\theoremstyle{definition}
  \newtheorem{definition}[subsection]{Definition}

\theoremstyle{remark}
  
  \newtheorem{remark}[subsection]{Remark}

\numberwithin{equation}{section}

\author{Yehao Zhou}
\address{Yehao Zhou\newline
\indent Perimeter Institute for Theoretical Physics,\newline
\indent 31 Caroline St N, Waterloo, ON, Canada N2L 2Y5;\newline
\indent Department of Physics and Astronomy,\newline
\indent University of Waterloo.
}
\email{yzhou3@pitp.ca}

\title{On the Reducedness of Quiver Schemes}

\begin{document}

\begin{abstract}
In this paper we prove that a quiver scheme in characteristic zero is reduced if the moment map is flat. We use the reducedness result to show that the equivariant integration formula computes the K-theoretic Nekrasov partition function of five dimensional $\mathcal N=2$ quiver gauge theories when the moment map is flat. We also give an explicit characterization of flatness of moment map for finite and affine type A Dynkin quivers with framings. As an application, we give a refinement of a theorem of Ivan Losev on the relation between quantized Nakajima quiver variety in type A and parabolic finite W-algebra in type A.
\end{abstract}

% We also present some applications of the reducedness result. 

\maketitle

\section{Introduction}
Nakajima's quiver varieties play important roles in representation theory, algebraic geometry, mathematical physics and other fields. For example they geometrically realize representations of Kac-Moody algebras \cite{nakajima1998quiver}, Yangians \cite{varagnolo2000quiver}, quantum affine algebras \cite{nakajima2001quiver}, and they are the key objects in the study of Bethe/gauge correspondences \cite{maulik2012quantum}. They also show up as Higgs branches of three dimensional $\mathcal N=4$ quiver gauge theories \cite{intriligator1996mirror}.

Quiver varieties can be defined analytically as hyper-K\" ahler reductions of ADHM data of a quiver \cite{nakajima1994instantons}, and also algebraically as a GIT quotient \cite{nakajima1994instantons}. The algebraic construction gives rise to a \textit{quiver scheme}, whose underlying reduced scheme structure is the quiver variety. The formal definition is as follows.

Let $(Q,\mathbf{v})$ be a quiver, where $Q$ is a finite directed graph with vertex set $Q_0$ and arrow set $Q_1$, and $\mathbf{v}\in \mathbb N^{Q_0}$ is called the dimension vector. Define the representation space of $(Q,\mathbf{v})$
\begin{align}
    R(Q,\vv)=\bigoplus_{a\in Q_1}\mathrm{Hom}(\mathbb C^{\mathbf{v}_{t(a)}},\mathbb C^{\mathbf{v}_{h(a)}}),
\end{align}
where $h(a)$ and $t(a)$ are head and tail of an arrow $a\in Q_1$. The reductive group $G(\vv)=\prod_{i\in Q_0}\mathrm{GL}(\mathbf{v}_i)/\mathbb C^{\times}$ naturally acts on $R(Q,\vv)$, where $\mathbb C^{\times}$ embeds as the diagonal center and it acts trivially. There is a moment map $\mu:T^*R(Q,\mathbf{v})\to \g(\mathbf{v})^*$ associated to the $\GL(\mathbf{v})$-action. Let $Z$ be the subspace of $\prod_{\vv_i\neq 0}\mathbb C\subset \mathbb C^{Q_0}$ that is annihilated by $\vv$. Note that $Z$ can be identified with $(\g(\vv)/[\g(\vv),\g(\vv)])^*$, which is a subspace of $\g(\vv)^*$. Similarly let $\Theta$ be the subspace of $\prod_{\vv_i\neq 0}\mathbb Q\subset\mathbb Q^{Q_0}$ that is annihilated by $\vv$. For $\theta\in \Theta$ , we denote the $\theta$-semistable (respectively $\theta$-stable) locus of $\mu^{-1}(Z)$ by $\mu^{-1}(Z)^{\theta\mathrm{-ss}}$ (respectively $\mu^{-1}(Z)^{\theta\mathrm{-s}}$). For the definition of $\theta$-(semi)stability, see \cite{king1994moduli}.

\begin{definition}
Define the universal quiver scheme with stability condition $\theta\in \Theta$ by the GIT quotient
\begin{align}
    \mathcal M^{\theta}_Z(Q,\vv):=\mu^{-1}(Z)^{\theta\mathrm{-ss}}/G(\vv).
\end{align}
And for $\lambda\in Z$, define the quiver scheme with stability condition $\theta$ by the GIT quotient
\begin{align}
    \mathcal M^{\theta}_{\lambda}(Q,\vv):=\mu^{-1}({\lambda})^{\theta\mathrm{-ss}}/G(\vv).
\end{align}
Here we do not take reduced scheme structure.
\end{definition}

In the study of using the cohomology of quiver schemes to geometrize representations of quantum algebras, it is fine to take the underlying reduced scheme structure, since this does not make a difference to the cohomology. However in other circumstances the potentially non-reduced scheme structure does raise an issue.

One example comes from the computation of K-theoretic Nekrasov partition function of 5d $\mathcal N=2$ quiver gauge theories, which is the equivariant K-theory class of $\Gamma(\mathcal M^{\theta}_{0}(Q,\vv),\mathcal O_{\mathcal M^{\theta}_{0}(Q,\vv)})$ for a generic stability parameter $\theta$ (see Definition \ref{def: generic stability}). Physicists proposed a way to compute it using the equivariant integration \cite{nekrasov2004abcd,butti2007counting,benvenuti2010hilbert}, which effectively computes the equivariant K-theory class of the differential graded algebra 
\begin{align}\label{eqn: dga}
    \left(\mathbb C[T^*R(Q,\vv)]\overset{L}{\otimes}_{\mathbb C[\mathfrak{g}(\vv)^*]}\mathbb C\right)^{G(\vv)},
\end{align}
via the Koszul resolution of $\mathbb C$ as a $\mathbb C[\mathfrak{g}(\vv)^*]$-module. There are two potential issues in this proposal. The first is that \eqref{eqn: dga} might have non-zero homological degree components, which maps to zero in $\Gamma(\mathcal M^{\theta}_{0}(Q,\vv),\mathcal O_{\mathcal M^{\theta}_{0}(Q,\vv)})$ via the natural projection $\mathcal M^{\theta}_{0}(Q,\vv)\to \mathcal M_{0}(Q,\vv)$. One may resolve this issue by requiring that the moment map $\mu:T^*R(Q,\vv)\to \mathfrak{g}(\vv)^*$ is flat, then \eqref{eqn: dga} is concentrated in homological degree zero, and this is in fact the case for many examples computed in \cite{benvenuti2010hilbert}. Another issue is that the $H^0$ part of \eqref{eqn: dga} is $\mathbb C[\mathcal M_0(Q,\vv)]$, which might not be isomorphic to $\Gamma(\mathcal M^{\theta}_{0}(Q,\vv),\mathcal O_{\mathcal M^{\theta}_{0}(Q,\vv)})$: $\mathcal M^{\theta}_{0}(Q,\vv)$ and $\mathcal M_{0}(Q,\vv)$ can have different dimensions, and even if they have the same dimension, they are not isomorphic if there are nilpotent elements in $\mathbb C[\mathcal M_0(Q,\vv)]$. This is the place where reducedness becomes important. Nevertheless, we will see that if the moment map is flat, then both of the potential issues are resolved, namely the differential graded algebra \eqref{eqn: dga} is quasi-isomorphic to $\Gamma(\mathcal M^{\theta}_{0}(Q,\vv),\mathcal O_{\mathcal M^{\theta}_{0}(Q,\vv)})$ (see Corollary \ref{cor: resolution of singularity}), in other words the equivariant integration formula holds under the assumption on the flatness of the moment map.

It seems to be unknown whether $\mathcal M^{\theta}_{\lambda}(Q,\vv)$ is in general reduced or not, though in some cases the reducedness have been shown. For example Gan and Ginzburg \cite{gan2006almost} showed that when $Q$ is a framed Jordan quiver, $\mathcal M^{\theta}_{\lambda}(Q,\vv)$ is reduced for all $(\theta,\lambda)\in \Theta\times Z$; Crawley-Boevey \cite{crawley2001geometry} showed that if $\vv\in \Sigma_{\lambda}$ then $\mathcal M_{\lambda}(Q,\vv)$ is reduced, later Bellamy and Schedler \cite{bellamy2016symplectic} extended Crawley-Boevey's result to that if $\vv\in \Sigma_{\lambda,\theta},\lambda\in \mathbb R^{Q_0}\cap Z$, then $\mathcal M^{\theta}_{\lambda}(Q,\vv)$ is reduced.

The main goal of this paper is to simultaneously generalize previous results on reducedness of quiver schemes of Gan and Ginzburg \cite{gan2006almost}, Crawley-Boevey \cite{crawley2001geometry}, Bellamy and Schedler \cite{bellamy2016symplectic}. The formal statement is as follows.

\begin{customthm}{A}\label{thm: reducedness_enhanced}
If the moment map $\mu$ is flat along $\mu^{-1}(\lambda)^{\theta\mathrm{-ss}}$, then the scheme $\mathcal{M}^{\theta}_{\lambda}(Q,\mathbf{v})$ is reduced.
\end{customthm}

When $Q$ is a framed Jordan quiver then $\mu$ is flat (cf. Proposition \ref{prop: flatness in type A}). When $\vv\in \Sigma_{\lambda,\theta}$ then $\mu$ is flat along $\mu^{-1}(\lambda)^{\theta\mathrm{-ss}}$ \cite[Proposition 3.28]{bellamy2016symplectic}. Therefore Theorem \ref{thm: reducedness_enhanced} indeed covers previous results on reducedness of quiver schemes in \cite{gan2006almost,crawley2001geometry,bellamy2016symplectic}. A special case of the above theorem is the following.

\begin{customthm}{B}\label{thm: reducedness}
If the moment map $\mu$ is flat, then $\forall (\theta,\lambda)\in \Theta\times Z$, the scheme $\mathcal{M}^{\theta}_{\lambda}(Q,\mathbf{v})$ is reduced.
\end{customthm}

In fact Theorem \ref{thm: reducedness_enhanced} can be deduced from Theorem \ref{thm: reducedness}, see Proposition \ref{prop: special case implies general case}.

\begin{remark}\label{rmk: connected components}
If $\mu$ is flat along $\mu^{-1}(\lambda)^{\theta\mathrm{-ss}}$ (assume it is non-empty), then the support of the dimension vector $\vv$ must be connected. In fact, if $\Supp(\vv)=\Supp(\vv^{(1)})\sqcup \Supp(\vv^{(2)})$ such that $\vv=\vv^{(1)}+\vv^{(2)}$, then the image of $\mu$ lies in $\mathfrak{g}(\vv^{(1)})^*\oplus \mathfrak{g}(\vv^{(2)})^*$, which is a proper linear subspace of $\mathfrak{g}(\vv)^{*}$, but the flatness of $\mu$ along $\mu^{-1}(\lambda)^{\theta\mathrm{-ss}}$ implies that the image of $\mu$ is Zariski-dense in $\mathfrak{g}(\vv)^{*}$, we get a contradiction. However, this is just a feature of our definition of the group action that it does not take the possibility of multiple connected components of $\Supp(\vv)$ into account, a more careful definition could separate the components and discuss flatness locally on each component. This is a straightforward generalization of the discussion in this paper, and in order to simplify notation we will not invoke this generalization throughout this paper.
\end{remark}

\begin{remark}\label{rmk: other Hamiltonian reductions}
One might wonder if Theorem \ref{thm: reducedness} generalizes to other Hamiltonian reductions, namely if $M$ is a complex symplectic representation of complex algebraic group $G$ with a flat moment map $\mu: M\to \mathfrak{g}^*$, then is $M\sslash_0 G=\mu^{-1}(0)/G$ reduced or not? The answer is that $M\sslash_0 G$ is not reduced in general, for example the algebraic Uhlenbeck partial compactification of moduli space of framed $\mathrm{SO}(3)$-instantons on $S^4$ with instanton number $4$ is isomorphic to a Hamiltonian reduction with flat moment map, but it is not reduced, see \cite{choy2016corrigendum,choy2016moduli} for more detail.
\end{remark}

The paper is organized as follows.

In Section \ref{sec: preliminaries} we present preliminaries on quiver schemes that will be useful in the proof of Theorem \ref{thm: reducedness_enhanced}, and we also derive some corollaries of Theorem \ref{thm: reducedness_enhanced}. Specifically we recall a criterion on the flatness of moment map in \ref{subsec: Criterion of flatness}, and recall a theorem on \' etale transversal slice in \ref{subsec: Transverse slice}, then reduces the proof of Theorem \ref{thm: reducedness_enhanced} to that of Theorem \ref{thm: reducedness}. In \ref{subsec: Generic parameters} and \ref{subsec: Relation between different stability conditions} we discuss generic stability parameters and deduce Lemma \ref{lemma: normality of family implies reducedness} which is a key to induction step in the proof of Theorem \ref{thm: reducedness}, then prove Corollary \ref{cor: resolution of singularity} which is useful in applications (for example the computation of K-theoretic Nekrasov partition functions). In \ref{subsec: Reflection isomorphisms} we apply the Theorem \ref{thm: reducedness} to construct the $(\pm 1)$-reflection isomorphism for all $(\theta,\lambda)\in \Theta\times Z$, assuming that the moment map is flat.

In Section \ref{sec: proof} we present the proof of Theorem \ref{thm: reducedness}. The idea is to use the induction on $|\vv|=\sum_{i\in Q_0}\vv_i$, combined with Lemma \ref{lemma: normality of family implies reducedness} and a result of Crawley-Boevey recalled in Lemma \ref{Lemma: Normality}. Then it boils down to some simple cases which are discussed in \ref{subsec: case of one vertex}, \ref{subsec: case of framed one vertex}, and the end of Section \ref{sec: proof}.

In Section \ref{sec: W-algebra} we discuss the application of Theorem \ref{thm: reducedness} to the study of quantization of quiver schemes. In \ref{subsec: affine type A} we focus on affine type A Dynkin quivers and give an explicit description of the set of $(\vv,\dd)$ such that the moment map for the framed quiver $(Q,\vv,\dd)$ is flat (Proposition \ref{prop: flatness in type A}). In \ref{subsec: Quantization of quiver schemes} we recall the definition of quantum Hamiltonian reduction and also its sheafified version, and show that in the quiver scheme case, if the moment map is flat and $\vv$ is indivisible then the quantized quiver scheme is isomorphic to certain sub-algebra of global sections of quantized structure sheaf on a resolution (Proposition \ref{prop: isomorphism different stability condition}). Finally in \ref{subsec: Finite W-algebras} we present a refinement of a theorem of Ivan Losev on the relation between quantized Nakajima quiver schemes in type A and parabolic finite W-algebras in type A (Proposition \ref{prop: quantum quiver scheme and W-algebra}).

\subsection*{Acknowledgement} I would like to thank Joel Kamnitzer, Seyed Faroogh Moosavian, and Hiraku Nakajima for helpful discussions and comments. Research at Perimeter Institute is supported by the Government of Canada through Industry Canada and by the Province of Ontario through the Ministry of Research and Innovation.

\section{Preliminaries on Quiver Schemes}\label{sec: preliminaries}

\subsection{Criterion of flatness of the moment map}\label{subsec: Criterion of flatness} We recall a theorem of Crawley-Boevey here.
\begin{lemma}[{\cite[Theorem 1.1]{crawley2001geometry}}]\label{lemma: flatness}
The moment map $\mu$ is flat if and only if
\begin{align}\label{eqn: goodness for framed quiver}
    \pp(\vv)\ge \sum_{t=1}^r\pp(\vv^{(t)}),
\end{align}
for all decomposition $\mathbf{v}=\mathbf{v}^{(1)}+\cdots+\mathbf{v}^{(r)}$ into non-zero elements in $\mathbb N^{Q_0}$. Here $\pp$ is the function
\begin{align}
    \pp(\vv)=1-\frac{1}{2}\vv\cdot C_Q\vv,
\end{align}
where $C_Q$ is the Cartan matrix of $Q$.
\end{lemma}

It is worth mentioning that \cite{su2006flatness} gives an explicit description of the set of $\vv$ such that the moment map is flat in terms of $(-1)$-reflections (see \textit{loc. cit.} for the definition). We will come back to this point at this end of this section.

\subsection{Transverse slice}\label{subsec: Transverse slice}
One important tool of our proof of Theorem \ref{thm: reducedness_enhanced} is the following result essentially due to Bellamy and Schedler \cite{bellamy2016symplectic} of which the $\theta=0$ case goes back to Crawley-Boevey \cite{crawley2001normality}:
\begin{lemma}\label{lem: transverse slice}
Take $\lambda\in Z$, then for every point $x\in \mu^{-1}(\lambda)^{\theta\mathrm{-ss}}$ such that its orbit $G(\vv)\cdot x$ is closed in $\mu^{-1}(\lambda)^{\theta\mathrm{-ss}}$, $G(\mathbf{v})\cdot x$ has an \' etale transverse slice in $x\in \mu^{-1}(\lambda)^{\theta\mathrm{-ss}}$ such that it is \' etale locally isomorphic to an open neighborhood $U$ of $0\in \hat\mu^{-1}(0)$, where $\hat\mu$ is the moment map for another quiver $(\hat{Q},\hat{\mathbf{v}})$ such that $G(\hat{\mathbf{v}})\cong G(\mathbf{v})_x$ (stabilizer of $x$). In addition, this \' etale transverse slice gives rise to an isomorphism:
\begin{align}
    \mathcal{M}^{\theta}_{\lambda}(Q,\mathbf{v})^{\wedge}_x   \cong \mathcal{M}_{0}(\hat Q,\hat\vv)^{\wedge}_0.
\end{align}
Here $(-)^{\wedge}_x$ means formal completion at $x$. Moreover if $\mu$ is flat at $x$, then $\hat\mu$ is flat.
\end{lemma}

\begin{proof}
The statement of \' etale transverse slice and local isomorphism are essentially \cite[Theorem 3.8]{bellamy2016symplectic}, and the argument in \textit{loc. cit.} works at the full scheme-theoretic level, i.e. not just for underlying reduced scheme structure. And we also emphasize that their argument works for general $\lambda\in Z$, not just for $\lambda\in \mathbb R^{Q_0}\cap Z$. We only need to show that if $\mu$ is flat at $x$, then $\hat\mu$ is flat.

Recall some of the ingredients of the construction of \' etale transverse slice. Since the stabilizer $G(\vv)_x$ is reductive by Matsushima's theorem \cite{luna1973slices}, there is $G(\vv)_x$-stable complement $L$ of $\mathfrak{g}(\vv)_x$ in $\mathfrak{g}(\vv)$.  By \cite[Lemma 4.1]{crawley2001normality}, the $G(\vv)_x$-submodule $\mathfrak{g}(\vv)\cdot x\subset T^*R(Q,\vv)$ is isotropic, and by \cite[Corollary 2.3]{crawley2001normality}, there exists a coisotropic $G(\vv)_x$-module complement $C$ to $\mathfrak{g}(\vv)\cdot x$ in $T^*R(Q,\vv)$. Let $W=(\mathfrak{g}(\vv)\cdot x)^{\perp}\cap C$. The composition of $\mu: T^*R(Q,\vv)\to \mathfrak{g}(\vv)^*$ with the restriction map $\mathfrak{g}(\vv)^*\to \mathfrak{g}(\vv)^*_x$ is denoted by $\mu_x$, and the restriction of $\mu_x$ to $W$ is denoted by $\hat{\mu}$. Then there is an identification $W\cong T^*R(\hat{Q},\hat{\vv})$, with the moment map being $\hat{\mu}$.

Note that the natural map $\eta: G(\vv)\times ^{G(\vv)_x} C\to T^*R(Q,\vv)$ is \' etale at $(1,0)\in G(\vv)\times ^{G(\vv)_x} C$, therefore the composition $\mu\circ\eta$ is flat at $(1,0)$. Since $\mu|_C$ agrees with the composition $\mu\circ\eta\circ \iota$, where $\iota:C=G(\vv)_x\times ^{G(\vv)_x} C\hookrightarrow G(\vv)\times ^{G(\vv)_x} C$ is the natural embedding, we see that $\mu|_C$ is flat at $0\in C$, thus $\mu_x|_C$ is flat at $0\in C$. According to \cite[Lemma 4.7]{crawley2001normality} we have $C=C^{\perp}\oplus W$ so there is a projection map $p:C\to W$, and it is easy to see that $\hat{\mu}\circ p=\mu_x|_C$, thus $\hat{\mu}$ is flat at $0\in W$. Then it follows that $\hat{\mu}$ is flat by the $\mathbb C^{\times}$-equivariance of $\hat{\mu}$, here the $\mathbb C^{\times}$ acts on the doubling of $\hat{Q}$ by scaling all arrows with weight one.

\end{proof}

\begin{remark}
$(\hat{Q},\hat{\mathbf{v}})$ is determined as following. Suppose that $x$ decomposes as
$$x=x_1^{\oplus k_1}\oplus\cdots\oplus x_r^{\oplus k_r}$$where $x_i$ are non-isomorphic $\theta$-stable representations, and let $\mathbf{v}^{(t)}$ be the dimension vector of $x_t$. Then the vertex set of $\hat{Q}$ is $\{1,\cdots,r\}$ with dimension vector $\hat{\mathbf{v}}_t=k_t,t\in \hat{Q}_0$
The adjacency matrix of $\hat{Q}$ is determined from its Euler form $(-,-)_{\hat{Q}}$, i.e. the sum of adjacency matrix and its transpose, which is 
\begin{align}
    (\hat\ee_t,\hat\ee_u)_{\hat{Q}}=2\delta_{tu}-\mathbf{v}^{(t)}\cdot C_Q\mathbf{v}^{(u)}.
\end{align}
Here $\hat\ee_t$ is the dimension vector on $\hat Q$ such that it is $1$ at vertex $t$ and zero elsewhere. As a corollary, we have
\begin{align}\label{eqn: equation of dimension}
    \hat\pp\left(\sum_{t}k_t\hat\ee_t\right)=\pp\left(\sum_{t}k_t\vv^{(t)}\right).
\end{align}
\end{remark}

\begin{lemma}\label{lem: characterization of flatness}
For a quiver $(Q,\vv)$, the reduced scheme $\mathcal{M}^{\theta}_{\lambda}(Q,\mathbf{v})_{\mathrm{red}}$ is normal for all $(\theta,\lambda)$. Moreover, following statements are equivalent:
\begin{itemize}
    \item[(1)] The moment map $\mu$ is flat along $\mu^{-1}(\lambda)^{\theta\mathrm{-ss}}$.
    \item[(2)] For every decomposition $\vv=\sum_{t=1}^rk_t\vv_t$ such that there exists mutually non-isomorphic $\theta$-stable representations $x_t$ of dimension $\vv_t$ and $\lambda\cdot\vv_t=0$, then 
    $$\pp(\vv)\ge \sum_{t=1}^r k_t\pp(\vv_t).$$
    \item[(3)] $\dim\mathcal{M}^{\theta}_{\lambda}(Q,\mathbf{v})=2\pp (\vv).$
\end{itemize}
\end{lemma}

\begin{proof}
In the view of Lemma \ref{lem: transverse slice}, for every point $x\in \mathcal{M}^{\theta}_{\lambda}(Q,\mathbf{v})$, the formal completion $\mathcal{M}^{\theta}_{\lambda}(Q,\mathbf{v})^{\wedge}_x $ is isomorphic to $\mathcal{M}_{0}(\hat Q,\hat\vv)^{\wedge}_0$ of some quiver $(\hat Q,\hat\vv)$, then the normality of $\mathcal{M}^{\theta}_{\lambda}(Q,\mathbf{v})_{\mathrm{red}}$ follows from \cite[Theorem 1.1]{crawley2001normality}.

$(1)\Rightarrow (3)$: For every point $x\in \mu^{-1}(\lambda)^{\theta\mathrm{-ss}}$ with closed $G(\vv)$-orbit, we attach the quiver $(\hat Q,\hat \vv)$ to it, and by Lemma \ref{lem: transverse slice} $\hat \mu$ is flat, so we have $$\dim \mathcal{M}^{\theta}_{\lambda}(Q,\mathbf{v})^{\wedge}_x=\dim\mathcal{M}_{0}(\hat Q,\hat\vv)^{\wedge}_0=2\hat\pp (\hat\vv)=2\pp (\vv).$$Here we use the equation \eqref{eqn: equation of dimension}.

$(3)\Rightarrow (2)$: Suppose there is a decomposition $\vv=\sum_{t=1}^rk_t\vv_t$ such that there exists mutually non-isomorphic $\theta$-stable representations $x_t$ of dimension $\vv_t$ and $\lambda\cdot\vv_t=0$, then let $x=x_1^{\oplus k_1}\oplus\cdots\oplus x_r^{\oplus k_r}$ be a point in $\mu^{-1}(\lambda)^{\theta\mathrm{-ss}}$ with closed $G(\vv)$-orbit, and we attach the quiver $(\hat Q,\hat \vv)$ to it, then we have $\dim\mathcal{M}_{0}(\hat Q,\hat\vv)^{\wedge}_0=2\hat\pp (\hat\vv)$, therefore $\dim\mathcal{M}_{0}(\hat Q,\hat\vv)=2\hat\pp (\hat\vv)$. By \cite[Theorem 1.3]{crawley2001geometry} we see that $\hat{\mu}$ is flat, thus by \cite[Theorem 1.1]{crawley2001geometry} we have
\begin{align*}
    \pp (\vv)=\hat\pp (\hat\vv)\ge \sum_{t=1}^rk_t \hat\pp (\hat\ee_t)=\sum_{t=1}^r k_t\pp(\vv_t).
\end{align*}

$(2)\Rightarrow (1)$: Generalizing \cite[Corollary 6.4]{crawley2001normality} to $\theta$-stable representations, of which the proof follows verbatim as \textit{loc. cit.}, we have
\begin{align*}
    \dim \xi^{-1}(\mathcal{M}^{\theta}_{\lambda}(Q,\mathbf{v})_{\tau})\le \vv\cdot\vv-1+\pp(\vv)+\sum_{t=1}^r\pp(\vv_i),
\end{align*}
Here $\tau=(k_1,\vv_1;\cdots;k_r,\vv_r)$ is a representation type such that $\vv_i\cdot\theta=\vv_i\cdot\lambda=0$, and $\mathcal{M}^{\theta}_{\lambda}(Q,\mathbf{v})_{\tau}$ is the locus of representation type $\tau$, and $\xi: \mu^{-1}(\lambda)^{\theta\mathrm{-ss}}\to \mathcal{M}^{\theta}_{\lambda}(Q,\mathbf{v})$ is the quotient map. Therefore we have $\dim\mu^{-1}(\lambda)^{\theta\mathrm{-ss}}\le \vv\cdot\vv-1+2\pp(\vv)$. On the other hand $\mu^{-1}(\lambda)$ is a fiber of morphism between two smooth schemes of relative dimension $\vv\cdot\vv-1+2\pp(\vv)$, every irreducible component of $\mu^{-1}(\lambda)$ has dimension $\ge \vv\cdot\vv-1+2\pp(\vv)$. Since $\mu^{-1}(\lambda)^{\theta\mathrm{-ss}}$ is open in $\mu^{-1}(\lambda)$, this forces $\dim\mu^{-1}(\lambda)^{\theta\mathrm{-ss}}= \vv\cdot\vv-1+2\pp(\vv)$, and therefore $\mu$ is flat along $\mu^{-1}(\lambda)^{\theta\mathrm{-ss}}$ by Miracle flatness theorem.
\end{proof}

An immediate consequence of Lemma \ref{lem: transverse slice} is the following. 
\begin{proposition}\label{prop: special case implies general case}
Theorem \ref{thm: reducedness} implies Theorem \ref{thm: reducedness_enhanced}.
\end{proposition}
\begin{proof}
Assume that Theorem \ref{thm: reducedness} holds. By Lemma \ref{lem: transverse slice} there is an isomorphism $\mathcal{M}^{\theta}_{\lambda}(Q,\mathbf{v})^{\wedge}_x   \cong \mathcal{M}_{0}(\hat Q,\hat\vv)^{\wedge}_0$ of schemes, and $\hat{\mu}$ is flat for the quiver $(\hat{Q},\hat{\vv})$ so $\mathcal{M}_{0}(\hat Q,\hat\vv)$ is reduced, and thus $\mathcal{M}^{\theta}_{\lambda}(Q,\mathbf{v})$ is reduced at $x$. Let $x$ runs through the set of points in $\mu^{-1}(\lambda)^{\theta\mathrm{-ss}}$ with closed $G(\vv)$-orbits, this set maps surjectively to $\mathcal{M}^{\theta}_{\lambda}(Q,\mathbf{v})$ by geometric invariant theory, hence $\mathcal{M}^{\theta}_{\lambda}(Q,\mathbf{v})$ is reduced.
\end{proof}

\subsection{Generic parameters}\label{subsec: Generic parameters}

\begin{definition}\label{def: generic stability}
We show that $(\theta,\lambda)\in \Theta\times Z$ is \textit{generic} if 
\begin{align}\label{eqn: generic stability}
    (\theta,\lambda)\in (\Theta\times Z)\bigg\backslash \bigcup _{\substack{\vv'\in \mathbb Z^{Q_0}\\ \vv\not\propto \vv'}}^{0<\mathbf{v}'<\mathbf{v}}\vv^{'\perp}.
\end{align}
Here the partial order on $\mathbb Z^{Q_0}$ is such that $\mathbf{v}^{(1)}\le\mathbf{v}^{(2)}$ if $\mathbf{v}^{(2)}-\mathbf{v}^{(1)}\in \mathbb N^{Q_0}$. The set of generic $(\theta,\lambda)$ is denoted by $(\Theta\times Z)_0$, and we use the notation $\Theta_0$ (resp. $Z_0$) to denote the intersection of $(\Theta\times Z)_0$ with $\Theta\times \{0\}$ (resp. $\{0\}\times Z$). If $\theta\in \Theta_0$ (resp. $\lambda\in Z_0$) we show that it is generic.
\end{definition}

\begin{remark}\label{rmk: smoothness}
It is well-known that $\mathcal{M}^{\theta}_{\lambda}(Q,\mathbf{v})$ is smooth if $(\theta,\lambda)$ is generic and $\vv$ is indivisible, i.e. $\nexists \vv'\in \mathbb N^{Q_0},k\in \mathbb Z_{>1}$ such that $\vv=k\vv'$. This can be proved as following. If $(\theta,\lambda)$ satisfies \eqref{eqn: generic stability}, then every $\theta$-semistable representation in $\mu^{-1}(\lambda)$ is $\theta$-stable, i.e.
$$\mu^{-1}(\lambda)^{\theta\mathrm{-ss}}=\mu^{-1}(\lambda)^{\theta\mathrm{-s}},$$
because the dimension vector of a proper sub-representation, denoted by $\mathbf{v}'$, must be orthogonal to $\lambda$, therefore $\theta\cdot\lambda\neq 0$, thus the representation is $\theta$-stable, and henceforth the stability implies that $\mu$ is smooth along $\mu^{-1}(\lambda)$ and the action of $G(\vv)$ is free. This proves the smoothness result.
\end{remark}

\subsection{Relation between different stability conditions}\label{subsec: Relation between different stability conditions}

Denote the quiver scheme with zero stability by $\mathcal M_{Z}(Q,\vv)$ and $\mathcal M_{\lambda}(Q,\vv)$. Then there is a projective morphism
\begin{align}\label{eqn: projective morphism}
    p^{\theta}: \mathcal{M}^{\theta}_{Z}(Q,\mathbf{v})\longrightarrow \mathcal{M}_{Z}(Q,\mathbf{v}).
\end{align}
\begin{lemma}
$p^{\theta}|_{Z_0}:\mathcal{M}^{\theta}_{Z}(Q,\mathbf{v})|_{Z_0}\longrightarrow \mathcal{M}_{Z}(Q,\mathbf{v})|_{Z_0}$ is isomorphism.
\end{lemma}
\begin{proof}
If $\lambda\in Z_0$, then every representation in $\mu^{-1}(\lambda)$ is automatically $\theta$-semistable.
\end{proof}

\begin{lemma}\label{lemma: generic case}
If the moment map $\mu$ is flat, and $(\theta,\lambda)$ is generic, and moreover assume that either $\lambda\in Z_0$ or $\lambda\in \mathbb R^{Q_0}$, then $\mathcal{M}^{\theta}_{\lambda}(Q,\mathbf{v})$ is reduced and irreducible.
\end{lemma}

\begin{proof}
If $\lambda\in Z_0$ then we can set $\theta$ equals to zero since $p^{\theta}: \mathcal{M}^{\theta}_{\lambda}(Q,\mathbf{v})\to \mathcal{M}_{\lambda}(Q,\mathbf{v})$ is isomorphism. Therefore we are in the setups of Bellamy and Schedler \cite[1.1]{bellamy2016symplectic}. We claim that $\vv\in \Sigma_{\lambda,\theta}$ (see \cite[Section 2]{bellamy2016symplectic} for notation), this means that for all decomposition $\mathbf{v}=\mathbf{v}^{(1)}+\cdots+\mathbf{v}^{(r)}$ into non-zero elements in $\mathbb N^{Q_0}$ such that $\forall t\in \{1,\cdots,r\}, \vv^{(t)}\cdot \lambda=\vv^{(t)}\cdot \theta=0$, the inequality
\begin{align}
    \pp(\vv)>\sum_{t=1}^r\pp(\vv^{(t)})
\end{align}
holds. Since $(\theta,\lambda)$ is generic by assumption, we only need to show the inequality when $\mathbf{v}$ is divisible and $\vv^{(t)}=k_t\ww$ for $k_t\in \mathbb N_{>0}$ and $\ww\in \mathbb N^{Q_0}$. By Lemma \ref{lemma: flatness}, we already have 
\begin{align*}
    \pp(\vv)\ge\sum_{t=1}^r\pp(\vv^{(t)}),
\end{align*}
and it remains to show that the inequality must be strict, i.e. equality never holds. In effect, we expand two sides the above inequality as functions of $\ww$:
\begin{align*}
    1-\frac{1}{2}\left(\sum_{t=1}^rk_t\right)^2\ww\cdot C_Q\ww\ge r-\frac{1}{2}\left(\sum_{t=1}^rk_t^2\right)\ww\cdot C_Q\ww,
\end{align*}
and conclude that $\ww\cdot C_Q\ww<0$. Since $\ww\cdot C_Q\ww$ is an even number, so $\ww\cdot C_Q\ww\le -2$. Hence we have
\begin{align*}
    \pp(\vv)-\sum_{t=1}^r\pp(\vv^{(t)})&\ge 1-r+\left(\sum_{t=1}^rk_t\right)^2-\left(\sum_{t=1}^rk_t^2\right)=1-r+2\sum_{1\le t<u\le r}k_tk_u\\
    &\ge 1-r+2{{r}\choose {2}}=(r-1)^2>0,
\end{align*}
and our claim is justified. Then \cite[Corollary 3.22]{bellamy2016symplectic} shows that the stable locus of $\mathcal{M}^{\theta}_{\lambda}(Q,\mathbf{v})$, denoted by $\mathcal{M}^{\theta}_{\lambda}(Q,\mathbf{v})^s$, is dense. And by \cite[Theorem 3.27]{bellamy2016symplectic}
\begin{align*}
    \dim \xi^{-1}(\mathcal{M}^{\theta}_{\lambda}(Q,\mathbf{v})\backslash \mathcal{M}^{\theta}_{\lambda}(Q,\mathbf{v})^s)< \dim \mu^{-1}(\lambda)^{\theta\mathrm{-ss}}.
\end{align*}
Here $\xi: \mu^{-1}(\lambda)^{\theta\mathrm{-ss}}\to \mathcal{M}^{\theta}_{\lambda}(Q,\mathbf{v})$ is the quotient map. Since $\mu^{-1}(\lambda)^{\theta\mathrm{-ss}}$ is Cohen-Macaulay hence equidimensional, this implies that the set of $\theta$-stable points in $\mu^{-1}(\lambda)^{\theta\mathrm{-ss}}$ is dense. Since $\mu$ is smooth at $\theta$-stable points, we conclude that $\mu^{-1}(\lambda)^{\theta\mathrm{-ss}}$ is generically smooth. Combined with the Cohen-Macaulay property, we see that $\mu^{-1}(\lambda)^{\theta\mathrm{-ss}}$ is reduced, thus $\mathcal{M}^{\theta}_{\lambda}(Q,\mathbf{v})$ is reduced.
\end{proof}

It is easy to see that $\forall \theta\in \Theta$ there exists $\theta'$ generic and in the Euclidean neighborhood of $\theta$ and such that 
\begin{align*}
    \mu^{-1}(Z)^{\theta'\mathrm{-ss}}\subset \mu^{-1}(Z)^{\theta\mathrm{-ss}}.
\end{align*}
Then there exists projective morphism
\begin{align}
    p^{\theta'}_{\theta}: \mathcal{M}^{\theta'}_{Z}(Q,\mathbf{v})\longrightarrow \mathcal{M}^{\theta}_{Z}(Q,\mathbf{v}).
\end{align}
and it is compatible with \eqref{eqn: projective morphism} since $p^{\theta'}=p^{\theta}\circ p^{\theta'}_{\theta}$. Note that $p^{\theta'}_{\theta}|_{Z_0}$ is isomorphism, therefore it is a birational morphism.

\begin{lemma}
If the moment map $\mu$ is flat along $\mu^{-1}(\lambda)^{\theta'\mathrm{-ss}}$, and moreover assume that  $\dim\mathcal{M}^{\theta'}_{\lambda}(Q,\mathbf{v})=\dim \mathcal{M}^{\theta}_{\lambda}(Q,\mathbf{v})$, then $\mu$ is flat along $\mu^{-1}(\lambda)^{\theta\mathrm{-ss}}$.
\end{lemma}

\begin{proof}
Apply the equivalence $(1)\Leftrightarrow (3)$ in Lemma \ref{lem: characterization of flatness}.
\end{proof}

\begin{lemma}\label{lemma: normality of family implies reducedness}
If the moment map $\mu$ is flat, and moreover assume that $\mathcal M^{\theta}_Z(Q,\vv)$ is normal, then $\mathcal{M}^{\theta}_{\lambda}(Q,\mathbf{v})$ is reduced if $\mathcal{M}^{\theta'}_{\lambda}(Q,\mathbf{v})$ is reduced, and if this is the case then the morphism $p^{\theta'}_{\theta}: \mathcal{M}^{\theta'}_{\lambda}(Q,\mathbf{v})\rightarrow \mathcal{M}^{\theta}_{\lambda}(Q,\mathbf{v})$ is birational.
\end{lemma}

\begin{proof}
Note that the normality of $\mathcal M^{\theta}_Z(Q,\vv)$ implies that $p^{\theta'}_{\theta}$ has connected fibers since it is an isomorphism over the open subset $Z_0$. In particular, the restriction of $p^{\theta'}_{\theta}$ to any fiber over $\lambda\in Z$ gives rise to a birational morphism $\mathcal{M}^{\theta'}_{\lambda}(Q,\mathbf{v})_{\mathrm{red}}\to\mathcal{M}^{\theta}_{\lambda}(Q,\mathbf{v})_{\mathrm{red}}$ since both sides have the same dimension and $p^{\theta'}_{\theta}$ has connected fibers. $\mathcal{M}^{\theta'}_{0}(Q,\mathbf{v})$ is reduced by Lemma \ref{lemma: generic case}, and it remains to show that $\mathcal{M}^{\theta}_{\lambda}(Q,\mathbf{v})$ is reduced.

We claim that $Rp^{\theta'}_{\theta*}\mathcal O_{\mathcal{M}^{\theta'}_{Z}(Q,\mathbf{v})}$ is concentrated in degree zero and $p^{\theta'}_{\theta*}\mathcal O_{\mathcal{M}^{\theta'}_{Z}(Q,\mathbf{v})}$ is flat over $Z$. In effect, since both $\mathcal{M}^{\theta'}_{Z}(Q,\mathbf{v})$ and $\mathcal{M}^{\theta}_{Z}(Q,\mathbf{v})$ are flat over $Z$ and $p^{\theta'}_{\theta}$ is proper, $Rp^{\theta'}_{\theta*}\mathcal O_{\mathcal{M}^{\theta'}_{Z}(Q,\mathbf{v})}$ is is a relative perfect complex, and its formation commutes with base change $Z'\to Z$. In particular, we base change to $i:0\hookrightarrow Z$, and get
\begin{align*}
    Li^*Rp^{\theta'}_{\theta*}\mathcal O_{\mathcal{M}^{\theta'}_{Z}(Q,\mathbf{v})}\cong Rp^{\theta'}_{\theta*} ({\mathcal{M}^{\theta'}_{0}(Q,\mathbf{v})},\mathcal O_{\mathcal{M}^{\theta'}_{0}(Q,\mathbf{v})}).
\end{align*}
The RHS is concentrated in degree zero, because both $\mathcal{M}^{\theta}_{0}(Q,\mathbf{v})_{\mathrm{red}}$ and $\mathcal{M}^{\theta'}_{0}(Q,\mathbf{v})$ (which is reduced by Lemma \ref{lemma: generic case}) have symplectic singularities \cite[Theorem 1.2]{bellamy2016symplectic}, in particular they have rational singularities, moreover the induced morphism between them is birational, so we conclude that $R^kp^{\theta'}_{\theta*} ({\mathcal{M}^{\theta'}_{0}(Q,\mathbf{v})},\mathcal O_{\mathcal{M}^{\theta'}_{0}(Q,\mathbf{v})})$ for $k>0$ (using the fact that any desingularization of $\mathcal{M}^{\theta'}_{0}(Q,\mathbf{v})$ is automatically a desingularization of $\mathcal{M}^{\theta}_{0}(Q,\mathbf{v})_{\mathrm{red}}$). Hence $Rp^{\theta'}_{\theta*}\mathcal O_{\mathcal{M}^{\theta'}_{Z}(Q,\mathbf{v})}$ is concentrated in degree zero in an open neighborhood of $\mathcal{M}^{\theta}_{0}(Q,\mathbf{v})$. Next we observe that there is a $\mathbb C^{\times}$ action on the quiver which scales every arrow with weight one, this action commutes with $\GL(\mathbf{v})$-action and descends to actions on $\mathcal{M}^{\theta'}_{Z}(Q,\mathbf{v})$ and $\mathcal{M}^{\theta}_{Z}(Q,\mathbf{v})$ and $p^{\theta'}_{\theta}$ is equivariant. Since $Z$ has positive weights under this $\mathbb C^{\times}$-action, $\mathcal{M}^{\theta}_{Z}(Q,\mathbf{v})$ contracts to $\mathcal{M}^{\theta}_{0}(Q,\mathbf{v})$ under this $\mathbb C^{\times}$-action, henceforth $Rp^{\theta'}_{\theta*}\mathcal O_{\mathcal{M}^{\theta'}_{Z}(Q,\mathbf{v})}$ is concentrated in degree zero on the whole $\mathcal{M}^{\theta}_{Z}(Q,\mathbf{v})$. $p^{\theta'}_{\theta*}\mathcal O_{\mathcal{M}^{\theta'}_{Z}(Q,\mathbf{v})}$ is flat over $Z$ because it's relatively perfect over $Z$.

The punchline of above discussions is that we have a homomorphism between sheaves of rings
\begin{align}\label{eqn: map of sheaves}
    \mathcal O_{\mathcal{M}^{\theta}_{Z}(Q,\mathbf{v})}\longrightarrow p^{\theta'}_{\theta*}\mathcal O_{\mathcal{M}^{\theta'}_{Z}(Q,\mathbf{v})},
\end{align}
and it is an isomorphism on the open locus $Z_0$. By the assumption that $\mathcal{M}^{\theta}_{Z}(Q,\mathbf{v})$ is normal, the homomorphism \eqref{eqn: map of sheaves} is isomorphism. Since $p^{\theta'}_{\theta*}\mathcal O_{\mathcal{M}^{\theta'}_{Z}(Q,\mathbf{v})}$ has base change property, specialization to arbitrary $\lambda\in Z$ gives us isomorphism $$\mathcal O_{\mathcal{M}^{\theta}_{\lambda}(Q,\mathbf{v})}\cong p^{\theta'}_{\theta*}\mathcal O_{\mathcal{M}^{\theta'}_{\lambda}(Q,\mathbf{v})}.$$ In particular, $\mathcal{M}^{\theta}_{\lambda}(Q,\mathbf{v})$ is reduced if $\mathcal{M}^{\theta'}_{\lambda}(Q,\mathbf{v})$ is reduced.
\end{proof}

An immediate consequence of Lemma \ref{lemma: normality of family implies reducedness} and Theorem \ref{thm: reducedness} is the following:
\begin{corollary}\label{cor: resolution of singularity}
If the moment map $\mu$ is flat, then $\forall (\theta, \lambda)\in \Theta\times Z$, the projective morphisms $p^{\theta'}_{\theta}: \mathcal{M}^{\theta'}_{\lambda}(Q,\mathbf{v})\rightarrow \mathcal{M}^{\theta}_{\lambda}(Q,\mathbf{v})$ are birational. In particular, we have a quasi-isomorphism of differential graded algebra
\begin{align}\label{eqn: quasi-isom}
    \left(\mathbb C[T^*R(Q,\vv)]\overset{L}{\otimes}_{\mathbb C[\mathfrak{g}(\vv)^*]}\mathbb C\right)^{G(\vv)}\cong \Gamma(\mathcal M^{\theta}_{0}(Q,\vv),\mathcal O_{\mathcal M^{\theta}_{0}(Q,\vv)}),
\end{align}
where $\mathbb C$ in the left hand side is the stalk of $\mathcal O_{\mathfrak{g}(\vv)^*}$ at zero.
\end{corollary}
According to the Introduction, the quasi-isomorphism \eqref{eqn: quasi-isom} shows that the equivariant integration formula in the physics literature \cite{nekrasov2004abcd,butti2007counting,benvenuti2010hilbert} indeed computes the equivariant K-theory class of $\Gamma(\mathcal M^{\theta}_{0}(Q,\vv),\mathcal O_{\mathcal M^{\theta}_{0}(Q,\vv)})$, i.e. the K-theoretic Nekrasov partition function.

\subsection{$(\pm 1)$-Reflection isomorphisms}\label{subsec: Reflection isomorphisms} Consider the reflections $s_i$ at loop-free vertex $\ee_i$, which acts on the quiver data $(\vv,\lambda,\theta)$ as
\begin{align}
    s_i\vv=\vv-(\vv,\ee_i)_Q\ee_i,\; (s_i\lambda)_j=\lambda_j-(\ee_i,\ee_j)_Q\lambda_i,\; (s_i\theta)_j=\theta_j-(\ee_i,\ee_j)_Q\theta_i.
\end{align}
The Lusztig-Maffei-Nakajima reflection isomorphism \cite[Theorem 26]{maffei2002remark} shows that if either $\lambda_i$ or $\theta_i$ is non-zero, then there is an isomorphism of schemes $\Phi_{s_i}:\mathcal{M}^{\theta}_{\lambda}(Q,\mathbf{v})\cong \mathcal{M}^{s_i\theta}_{s_i\lambda}(Q,s_i\mathbf{v})$. It is implicit in \cite{maffei2002remark} that the construction there actually holds for the full scheme structure, not just for reduced scheme structure. The key ingredient in the construction is an auxiliary scheme $Z^{\theta,\lambda}_i$ with projections 
\begin{center}
\begin{tikzcd}
\mu_{\vv}^{-1}(\lambda)^{\theta\mathrm{-ss}} & Z^{\theta,\lambda}_i \arrow[l,"p_i" ']  \arrow[r,"p'_i"] & \mu_{s_i\vv}^{-1}(s_i\lambda)^{s_i\theta\mathrm{-ss}},
\end{tikzcd}
\end{center}
such that $p_i$ is $\GL((s_i\vv)_i)$-torsor, and $p'_i$ is $\GL(\vv_i)$-torsor. Moreover, there is an action of $G_i(\vv)=\prod_{j\neq i}\GL(\vv_j)\times \GL(\vv_i)\times \GL((s_i\vv)_i)$ on $Z^{\theta,\lambda}_i$ such that $\GL((s_i\vv)_i)$ acts through the torsor $p_i$ and $\GL(\vv_i)$ acts through the torsor $p'_i$, and furthermore $p_i$ and $p'_i$ are $G_i(\vv)$-equivariant. Passing to the categorical quotient by $G_i(\vv)$, there are two isomorphisms:
\begin{center}
\begin{tikzcd}
\mathcal{M}^{\theta}_{\lambda}(Q,\mathbf{v}) & Z^{\theta,\lambda}_i/G_i(\vv) \arrow[l,"\overline p_i" ']  \arrow[r,"\overline p'_i"] & \mathcal{M}^{s_i\theta}_{s_i\lambda}(Q,s_i\mathbf{v}),
\end{tikzcd}
\end{center}
and we set $\Phi_{s_i}=\overline p'_i\circ \overline p_i^{-1}$. The condition that either $\lambda_i$ or $\theta_i$ is non-zero is crucial in the construction of $Z^{\theta,\lambda}_i$. In the following we will eliminate this condition, but only for flat $\mu_{\vv}$ and $(\pm 1)$-reflections.

\begin{theorem}\label{thm: reflection isom}
If the moment map $\mu$ is flat, and $i\in Q_0$ is a loop-free vertex such that $(\vv,\ee_i)_Q=\pm 1$, then $\forall (\theta,\lambda)\in \Theta\times Z$, there is an isomorphism $\Phi^{\theta,\lambda}_{s_i}:\mathcal{M}^{\theta}_{\lambda}(Q,\mathbf{v})\cong \mathcal{M}^{s_i\theta}_{s_i\lambda}(Q,s_i\mathbf{v})$ such that the diagram
\begin{center}
\begin{tikzcd}
\mathcal{M}^{\theta}_{\lambda}(Q,\mathbf{v}) \arrow[d]  \arrow[r,"\Phi^{\theta,\lambda}_{s_i}"] & \mathcal{M}^{s_i\theta}_{s_i\lambda}(Q,s_i\mathbf{v}) \arrow[d]\\
\mathcal{M}_{\lambda}(Q,\mathbf{v})  \arrow[r,"\Phi^{0,\lambda}_{s_i}"] & \mathcal{M}_{s_i\lambda}(Q,s_i\mathbf{v}) 
\end{tikzcd}
\end{center}
commutes. Moreover if either $\lambda_i$ or $\theta_i$ is non-zero then $\Phi^{\theta,\lambda}_{s_i}$ agrees with the Lusztig-Maffei-Nakajima reflection isomorphism.
\end{theorem}

\begin{proof}
To begin with, we note that the moment map for $(Q,s_i\vv)$ is flat by Proposition \cite[Proposition 7.1]{su2006flatness}. If either $\lambda_i$ or $\theta_i$ is non-zero then we define $\Phi^{\theta,\lambda}_{s_i}$ to be the Lusztig-Maffei-Nakajima reflection isomorphism. Note that in this case if $\lambda_i\neq 0$, the construction of $Z^{\theta,\lambda}_i$ fits into a commutative diagram
\begin{center}
\begin{tikzcd}
\mu_{\vv}^{-1}(\lambda)^{\theta\mathrm{-ss}}\arrow[d,hook] & Z^{\theta,\lambda}_i \arrow[l,"p_i" ']  \arrow[r,"p'_i"] \arrow[d,hook] & \mu_{s_i\vv}^{-1}(s_i\lambda)^{s_i\theta\mathrm{-ss}} \arrow[d,hook]\\
\mu_{\vv}^{-1}(\lambda) & Z^{0,\lambda}_i \arrow[l,"p_i" ']  \arrow[r,"p'_i"] & \mu_{s_i\vv}^{-1}(s_i\lambda)
\end{tikzcd}
\end{center}
Thus the diagram 
\begin{center}
\begin{tikzcd}
\mathcal{M}^{\theta}_{\lambda}(Q,\mathbf{v}) \arrow[d]  \arrow[r,"\Phi^{\theta,\lambda}_{s_i}"] & \mathcal{M}^{s_i\theta}_{s_i\lambda}(Q,s_i\mathbf{v}) \arrow[d]\\
\mathcal{M}_{\lambda}(Q,\mathbf{v})  \arrow[r,"\Phi^{0,\lambda}_{s_i}"] & \mathcal{M}_{s_i\lambda}(Q,s_i\mathbf{v}) 
\end{tikzcd}
\end{center}
commutes.

If $\lambda_i=\theta_i=0$, we define $\Phi^{\theta,\lambda}_{s_i}$ as following. First we find a generic $\theta'$ in the Euclidean neighborhood of $\theta$ such that $\mu^{-1}(\lambda)^{\theta'\mathrm{-ss}}\subset \mu^{-1}(\lambda)^{\theta\mathrm{-ss}}$, then by \ref{cor: resolution of singularity} the morphism $p^{\theta'}_{\theta}:\mathcal{M}^{\theta'}_{\lambda}(Q,\mathbf{v})\to \mathcal{M}^{\theta}_{\lambda}(Q,\mathbf{v})$ is projective and birational. Since $\theta'$ is generic, in particular $\theta_i'\neq 0$, we have Lusztig-Maffei-Nakajima reflection isomorphism $\Phi^{\theta',\lambda}_{s_i}: \mathcal{M}^{\theta'}_{\lambda}(Q,\mathbf{v})\cong \mathcal{M}^{s_i\theta'}_{s_i\lambda}(Q,s_i\mathbf{v})$. Taking affinization of the domain and codomain of $\Phi^{\theta',\lambda}_{s_i}$, we get an isomorphism $\Phi^{0,\lambda}_{s_i}: \mathcal{M}_{\lambda}(Q,\mathbf{v})\cong \mathcal{M}_{s_i\lambda}(Q,s_i\mathbf{v})$.

After scaling by an integer, we assume that $\theta\in \mathbb Z^{Q_0}$. By the construction of GIT quotient, there is an ample line bundle $\mathcal L(\theta)$ on $\mathcal{M}^{\theta}_{\lambda}(Q,\mathbf{v})$ such that 
\begin{align}
    \mathcal{M}^{\theta}_{\lambda}(Q,\mathbf{v})=\Proj \bigoplus_{n\ge 0}\Gamma(\mathcal{M}^{\theta}_{\lambda}(Q,\mathbf{v}),\mathcal L(\theta)^{\otimes n}),
\end{align}
as a scheme over $\Spec \Gamma(\mathcal{M}^{\theta}_{\lambda}(Q,\mathbf{v}),\mathcal O_{\mathcal{M}^{\theta}_{\lambda}(Q,\mathbf{v})})=\mathcal{M}_{\lambda}(Q,\mathbf{v})$. Similarly there is an ample line bundle $\mathcal L(s_i\theta)$ on $\mathcal{M}^{s_i\theta}_{s_i\lambda}(Q,s_i\mathbf{v})$ with similar property. We claim that 
\begin{align}\label{eqn: isom of line bundle}
    p^{\theta'*}_{\theta}(\mathcal L(\theta))\cong (\Phi^{\theta',\lambda}_{s_i})^*p^{s_i\theta'*}_{s_i\theta}(\mathcal L(s_i\theta))
\end{align}
In effect, the pullback of $\mathcal L(\theta)$ to $\mu^{-1}(\lambda)^{\theta\mathrm{-ss}}$ is the $G(\vv)$-equivariant line bundle $\prod_{j\in Q_0}\det(V_j^*)^{\otimes \theta_j}$, where $V_j$ is the vector bundle on $T^*R(Q,\vv)$ with fibers being the vector space at $j$'th vertex. Similar fact holds for the pullback of $\mathcal L(s_i\theta)$ to $\mu^{-1}(s_i\lambda)^{s_i\theta\mathrm{-ss}}$. By the construction of $Z^{\theta',\lambda}_i$ and the fact that $\theta_i=(s_i\theta)_i=0$, we have 
\begin{align*}
    p_i^*\xi^*(\mathcal L(\theta))=\prod_{j\neq i}\det(V_j^*)^{\otimes \theta_j}= p'^*_i\xi'^*(\mathcal L(s_i\theta)).
\end{align*}
Here $\xi: \mu^{-1}(\lambda)^{\theta\mathrm{-ss}}\to \mathcal{M}^{\theta}_{\lambda}(Q,\mathbf{v}),\xi':\mu^{-1}(s_i\lambda)^{s_i\theta\mathrm{-ss}}\to \mathcal{M}^{s_i\theta}_{s_i\lambda}(Q,s_i\mathbf{v})$ are quotient maps. Therefore the claim follows. Since $p^{\theta'}_{\theta}$ is a proper birational morphism between normal schemes, we have 
\begin{align}\label{eqn: proj spectra}
    \mathcal{M}^{\theta}_{\lambda}(Q,\mathbf{v})=\Proj \bigoplus_{n\ge 0}\Gamma(\mathcal{M}^{\theta'}_{\lambda}(Q,\mathbf{v}),p^{\theta'*}_{\theta}\mathcal L(\theta)^{\otimes n}).
\end{align}
Similar fact holds for $\mathcal{M}^{s_i\theta}_{s_i\lambda}(Q,s_i\mathbf{v})$. Combining \eqref{eqn: proj spectra} with \eqref{eqn: isom of line bundle}, we get an isomorphism $\Phi^{\theta,\lambda}_{s_i}:\mathcal{M}^{\theta}_{\lambda}(Q,\mathbf{v})\cong \mathcal{M}^{s_i\theta}_{s_i\lambda}(Q,s_i\mathbf{v})$. Obviously the diagram
\begin{center}
\begin{tikzcd}
\mathcal{M}^{\theta}_{\lambda}(Q,\mathbf{v}) \arrow[d]  \arrow[r,"\Phi^{\theta,\lambda}_{s_i}"] & \mathcal{M}^{s_i\theta}_{s_i\lambda}(Q,s_i\mathbf{v}) \arrow[d]\\
\mathcal{M}_{\lambda}(Q,\mathbf{v})  \arrow[r,"\Phi^{0,\lambda}_{s_i}"] & \mathcal{M}_{s_i\lambda}(Q,s_i\mathbf{v}) 
\end{tikzcd}
\end{center}
commutes. 

We claim that $\Phi^{\theta,\lambda}_{s_i}$ does not depend on the choice of $\theta'$. If $\lambda_i\neq 0$ then $\Phi^{\theta,\lambda}_{s_i}$ is just Lusztig-Maffei-Nakajima reflection isomorphism, so we assume that $\lambda_i=0$. In view of the above commutative diagram, it suffices to show that $\Phi^{0,\lambda}_{s_i}$ does not depend on the choice of $\theta'$. Consider the schemes
\begin{align*}
    \mathcal S_{n,m}&=\{X_1,Y_1,X_2,Y_2\:|\: X_1Y_1=0,X_2Y_2=0,Y_1X_1=Y_2X_2\}\\
    &\subset \mathrm{Mat}(n\times (n+m))\times \mathrm{Mat}( (n+m)\times n)\times \mathrm{Mat}(m\times (n+m))\times \mathrm{Mat}((n+m)\times n))\\
    \mathcal A_{n,m}&=\{X_1,Y_1\:|\: X_1Y_1=0\}\subset \mathrm{Mat}(n\times (n+m))\times \mathrm{Mat}( (n+m)\times n)\\
    \mathcal B_{n,m}&=\{X_2,Y_2\:|\: X_2Y_2=0\}\subset \mathrm{Mat}(m\times (n+m))\times \mathrm{Mat}( (n+m)\times m),
\end{align*}
together with obvious projections:
\begin{center}
\begin{tikzcd}
\mathcal A_{n,m} & \mathcal S_{n,m} \arrow[l,"\mathbf p_i" ']  \arrow[r,"\mathbf p'_i"] & \mathcal B_{n,m}.
\end{tikzcd}
\end{center}
The projections are equivariant under the natural $\GL(n)\times \GL(m)\times \GL(n+m)$ actions on $\mathcal S_{n,m},\mathcal A_{n,m},\mathcal B_{n,m}$. Define the quiver $Q^{(i)}$ as the sub-quiver of $Q$ deleting the vertex $i$ and all arrows $a$ such that $h(a)=i$ or $t(a)=i$. And let $\vv^{(i)}$ be the dimension vector of $Q^{(i)}$ such that $\vv^{(i)}_j=\vv_j$. Let $n=\vv_i,m=(s_i\vv)_i$, then it is easy to see that there are closed emdeddings
\begin{align*}
    \mu_{\vv}^{-1}(\lambda)\subset T^*R(Q^{(i)},\vv^{(i)})&\times \mathcal A_{n,m}\subset T^*R(Q,\vv),\\ 
    \mu_{s_i\vv}^{-1}(s_i\lambda)\subset T^*R(Q^{(i)},\vv^{(i)})&\times \mathcal B_{n,m}\subset T^*R(Q,s_i\vv).
\end{align*}
Consider the closed subscheme $Z^{\lambda}_i=\mathbf p_i^{-1}(\mu_{\vv}^{-1}(\lambda))\subset T^*R(Q^{(i)},\vv^{(i)})\times \mathcal S_{n,m}$, it is easy to see that $Z^{\lambda}_i$ is the same as $\mathbf p'^{-1}_i(\mu_{s_i\vv}^{-1}(s_i\lambda))$, then the natural projections
\begin{center}
\begin{tikzcd}
\mu_{\vv}^{-1}(\lambda) & Z^{\lambda}_i \arrow[l]  \arrow[r] & \mu_{s_i\vv}^{-1}(s_i\lambda)
\end{tikzcd}
\end{center}
are $G_i(\vv)$-equivariant. By the construction of \cite[Definition 27]{maffei2002remark}, we see that $Z^{\theta',\lambda}_i$ is an open subscheme of $Z^{\lambda}_i$. Therefore we have a sequence of maps of rings:
\begin{align*}
    \mathbb C[\mathcal M_{\lambda}(Q,\vv)]\to \mathbb C[Z^{\lambda}_i]^{G_i(\vv)}\to \Gamma(Z^{\theta',\lambda}_i,\mathcal O_{Z^{\theta',\lambda}_i})^{G_i(\vv)}=\Gamma(\mathcal M^{\theta'}_{\lambda}(Q,\vv),\mathcal O_{\mathcal M^{\theta'}_{\lambda}(Q,\vv)}).
\end{align*}
Since the composition $\mathbb C[\mathcal M_{\lambda}(Q,\vv)]\to \Gamma(\mathcal M^{\theta'}_{\lambda}(Q,\vv),\mathcal O_{\mathcal M^{\theta'}_{\lambda}(Q,\vv)})$ is isomorphism by Corollary \ref{cor: resolution of singularity}, we see that the first map $\mathbb C[\mathcal M_{\lambda}(Q,\vv)]\to \mathbb C[Z^{\lambda}_i]^{G_i(\vv)}$ is injective. On the other hand, the invariant theory shows that $G_i(\vv)$-invariant subalgebra of $\mathbb C[Z^{\lambda}_i]$ is generated by closed paths with arrows either in the doubled quiver $\overline Q$ or $X_2$ or $Y_2$, the equation $X_2Y_2=0,Y_1X_1=Y_2X_2$ imply that $\mathbb C[Z^{\lambda}_i]^{G_i(\vv)}$ is generated by $\mathbb C[\mu_{\vv}^{-1}(\lambda)]^{G_i(\vv)}$, so the map $\mathbb C[\mathcal M_{\lambda}(Q,\vv)]\to \mathbb C[Z^{\lambda}_i]^{G_i(\vv)}$ is surjective. We conclude that the natural morphism $\overline{\mathbf{p}}_i:Z^{\lambda}_i/^{G_i(\vv)}\to \mathcal M_{\lambda}(Q,\vv)$ is isomorphism, and similarly the natural morphism $\overline{\mathbf{p}}_i':Z^{\lambda}_i/^{G_i(\vv)}\to \mathcal M_{s_i\lambda}(Q,s_i\vv)$ is isomorphism. From the above argument we have $\Phi^{0,\lambda}_{s_i}=\overline{\mathbf{p}}_i'\circ \overline{\mathbf{p}}_i^{-1}$, and the latter does not depend on the choice of $\theta'$, therefore $\Phi^{0,\lambda}_{s_i}$ does not depend on the choice of $\theta'$.
\end{proof}

\section{Proof of Theorem \ref{thm: reducedness}}\label{sec: proof}

In this section we complete the proof of Theorem \ref{thm: reducedness}, thus proving Theorem \ref{thm: reducedness_enhanced} by Proposition \ref{prop: special case implies general case}. Since passing from $Q$ to the support of dimension vector $\vv$ does not affect the flatness, we will assume that $\vv_i>0$ for all $i\in Q_0$ in this section\footnote{Dimension vector such that $\vv_i>0$ for all $i\in Q_0$ is called \textit{sincere}, cf. \cite{crawley2001geometry}.}.

\subsection{Case of $|Q_0|=1$}\label{subsec: case of one vertex}
In this case the vertex set is just one element $Q_0=\{1\}$, and $\Theta=Z=\{0\}$. Suppose that there is only one arrow, then computation using Lemma \ref{lemma: flatness} shows that the only value of $\vv_1$ such that the moment map $\mu$ is flat is $\vv_1=1$, and in this case Theorem \ref{thm: reducedness} trivially holds. Suppose that there are more than one arrows, then an easy computation shows that 
\begin{align*}
    \pp(\vv)> \sum_{t=1}^r\pp(\vv^{(t)}),
\end{align*}
for all decomposition $\mathbf{v}=\mathbf{v}^{(1)}+\cdots+\mathbf{v}^{(r)}$ into non-zero elements in $\mathbb N^{Q_0}$ and $r>1$. In this case $\vv\in \Sigma_0$ (see comments after \cite[Theorem 1.2]{crawley2001geometry} for notation), and in this case Crawley-Boevey shows that $\mu^{-1}(0)$ is reduced and irreducible \cite[Theorem 1.2]{crawley2001geometry}, thus $\mathcal M_0(Q,\vv)$ is reduced hence normal, by \cite[Theorem 1.1]{crawley2001normality}. This finishes the proof of the case when $|Q_0|=1$.

\subsection{Case of $Q_0=\{1,2\}$ and $\vv_1=1$}\label{subsec: case of framed one vertex}
In this case $\vv_2=n>0$ and $(\theta,\lambda)\in \mathbb Q\times \mathbb C$. The main result of this subsection is following:

\begin{proposition}\label{prop: reducedness in rank 1}
If $Q_0=\{1,2\}$ and $\vv_1=1$, then $\mu^{-1}(0)$ is reduced.
\end{proposition}

\begin{proof}
Note that edge loop at vertex $1$ only contributes a vector space to $\mu^{-1}(0)$, without loss of generality we can assume that there is no edge loop at vertex $1$. Since the quiver will be doubled when considering moment map, without loss of generality we can assume that all arrows between vertices $1$ and $2$ are pointing from $1$ to $2$, and denote the number of such arrows by $k$. We split the discussion into three situations.
\begin{itemize}
    \item[(1)] $\mathbf v\in \Sigma_{0}$ (see comments after \cite[Theorem 1.2]{crawley2001geometry} for notation), this condition is equivalent to that the inequality in Lemma \ref{lemma: flatness} is strict, i.e. the equality never holds. In this situation \cite[Theorem 1.2]{crawley2001geometry} shows that $\mu^{-1}(0)$ is reduced and irreducible. An easy computation shows that $\mathbf v\in \Sigma_{0}$ exactly when there are more than two edge loops at vertex $2$, or there is one edge loop at vertex $2$ and $k>1$, or there is no edge loop at vertex $2$ and $k\ge 2n$.
    
    \item[(2)] If there is no edge loop in $Q$, then the flatness of $\mu$ is equivalent to $k\ge 2n-1$ (see Lemma \ref{lemma: flatness}). The case $k\ge 2n$ has been discussed above. If $k=n=1$, then straightforward computation finds that $\mu^{-1}(0)$ is reduced. The remaining cases are $k= 2n-1>1$.
    
    \item[(3)] If there is one edge loop and and $k=1$, then by the main result of Gan-Ginzburg \cite{gan2006almost}, $\mu^{-1}(0)$ is reduced in this case.
\end{itemize}
The idea of proof in the remaining cases in (2) is to show that every irreducible component of $\mu^{-1}(0)$ is generically reduced, then flatness of $\mu$ implies the Cohen-Macaulay property of $\mu^{-1}(0)$, which in turn implies that $\mu^{-1}(0)$ is reduced.\\
    
\textbf{Remaining part of situation (2).} We claim that $\mu^{-1}(0)$ has exactly two irreducible components. In effect consider the morphism $\pi: \mu^{-1}(0)\to R(Q,\mathbf{v})$ defined by forgetting the cotangent direction, then \cite[Lemma 4.3]{crawley2001geometry} implies that $\mu^{-1}(0)$ is union of constructible subsets of which the maximal dimensional ones are $\pi^{-1}I(\mathbf{v}^{(1)},\cdots,\mathbf{v}^{(r)})$, where $\mathbf{v}=\mathbf{v}^{(1)}+\cdots+\mathbf{v}^{(r)}$ is a decomposition of $\mathbf{v}$ into non-zero elements in $\mathbb N^{Q_0}$ such that $\mathbf{v}^{(1)}_1=1$, and 
\begin{align}\label{eqn: component}
    \pp(\mathbf{v})=\pp(\mathbf{v}^{(1)})+\cdots+\pp(\mathbf{v}^{(r)}).
\end{align}
Here $I(\mathbf{v}^{(1)},\cdots,\mathbf{v}^{(r)})$ is the constructible subset of $R(Q,\mathbf{v})$ consisting of the representations whose indecomposable summands have dimension $\mathbf{v}^{(1)},\cdots,\mathbf{v}^{(r)}$. Under the assumption that $k= 2n-1>1$, the only decompositions such that the equation \eqref{eqn: component} hold are $(\mathbf{v}^{(1)},\mathbf{v}^{(2)})=(\vv-\ee_2,\ee_2)$ and $\mathbf{v}^{(1)}=\mathbf{v}$. Here $\ee_2$ is the dimension vector $\ee_{2i}=\delta_{i2}$. Note that both $I(\mathbf{v}-\ee_2,\ee_2)$ and $I(\mathbf{v})$ consist of a single $\GL(n)\times \GL(k)$-orbit (by Gauss elimination), therefore preimages of both $\pi^{-1}I(\mathbf{v}-\ee_2,\ee_2)$ and $\pi^{-1}I(\mathbf{v})$ are irreducible by \cite[Lemma 3.4]{crawley2001geometry}, and $$\dim \mu^{-1}(0)\backslash (\pi^{-1}I(\mathbf{v}-\ee_2,\ee_2)\cup \pi^{-1}I(\mathbf{v}))<\dim \mu^{-1}(0).$$ $\mu^{-1}(0)$ is Cohen-Macaulay, so it is equidimensional, thus $$\mu^{-1}(0)=\overline{\pi^{-1}I(\mathbf{v}-\ee_2,\ee_2)}\cup \overline{\pi^{-1}I(\mathbf{v})}.$$ Notice that $\pi^{-1}I(\mathbf{v})$ contains $\theta$-stable representations for $(\theta_1,\theta_2)=(-n,1)$, thus $\pi^{-1}I(\mathbf{v})$ is generically smooth.

For $\pi^{-1}I(\mathbf{v}-\ee_2,\ee_2)$, consider a simple representation $x$ of  $(Q,\mathbf v-\ee_2)$ (simple representation exists for dimension vector $\vv-\ee_2$ by our previous discussion in the situation (1)) and take its direct sum with a trivial representation and regarded as a representation of $(Q,\mathbf v)$, still denoted by $x$, then $x$ is semisimple and $\pi(x)\in I(\mathbf{v}-\ee_2,\ee_2)$. According to Lemma \ref{lem: transverse slice}, there is an \' etale transverse slice of $G(\mathbf{v})\cdot x$ in $\mu^{-1}(0)$, which is \' etale locally isomorphic to an open neighborhood $U$ of $0\in \hat\mu^{-1}(0)$, where $\hat\mu$ is the classical moment map for the quiver $(\hat Q,\hat{\mathbf{v}})$. Here $\hat Q$ is the quiver with vertex set $\hat Q_0=\{1,2\}$, and the number of arrows between vertices $1$ and $2$ is $k+1-n$, which equals to $n>1$, and $\hat{\mathbf{v}}_1=\hat{\mathbf{v}}_2=1$. Note that $\hat{\mathbf{v}}\in \Sigma_0$ and this reduces to the situation (1), thus $\mu^{-1}(0)$ is reduced in an open neighborhood of $x$. Finally $x$ is in the irreducible component $\overline{\pi^{-1}I(\mathbf{v}-\ee_2,\ee_2)}$, so the irreducible component $\overline{\pi^{-1}I(\mathbf{v}-\ee_2,\ee_2)}$ is generically reduced.
\end{proof}

\begin{corollary}\label{cor: thm in rank 1}
If $Q_0=\{1,2\}$ and $\vv_1=1$, then Theorem \ref{thm: reducedness} holds.
\end{corollary}

\begin{proof}
If either $\theta$ or $\lambda$ is non-zero, then $(\theta,\lambda)$ satisfies \eqref{eqn: generic stability}, note that $\vv$ is obviously indivisible, thus $\mathcal M^{\theta}_{\lambda}(Q,\vv)$ and $\mathcal M^{\theta}_{Z}(Q,\vv)$ are smooth by Remark \ref{rmk: smoothness}. So we focus on the case $\theta=\lambda=0$. $\mu^{-1}(0)$ is reduced by Proposition \ref{prop: reducedness in rank 1}, and since $\mu^{-1}(Z)$ is Cohen-Macaulay and $\mu^{-1}(Z\backslash 0)$ is smooth, we conclude that $\mu^{-1}(Z)$ is normal, and henceforth $\mathcal M_Z(Q,\vv)$ is normal. By Lemma \ref{lemma: normality of family implies reducedness}, $\mathcal M_{0}(Q,\vv)$ is reduced.
\end{proof}

\subsection{General cases}
The last ingredient we need from Crawley-Boevey is the following technical lemma:
\begin{lemma}[{\cite[Corollary 7.2]{crawley2001normality}}]\label{Lemma: Normality}
Let X be a scheme over $\mathbb C$ with a reductive group $G$-action, assume that
\begin{itemize}
    \item[(1)] Categorical quotient $q:X\to X/G$ exists and $q$ is affine,
    \item[(2)] There is an open $U\subset X/G$ such that $U$ is normal, and its complement $Y$ has codimension $\ge 2$ in $X/G$,
    \item[(3)] $X$ has property $(S_2)$ and $q^{-1}(Y)$ has codimension $\ge 2$ in X.
\end{itemize}
Then $X/G$ is normal.
\end{lemma}

\begin{proof}[Proof of Theorem \ref{thm: reducedness}]
We assume that $Q$ is connected, otherwise the moment map can not be flat (see Remark \ref{rmk: connected components}). Define a partial order $\prec$ on the set of $(Q,\vv)$ by
\begin{center}
    $(Q',\vv')\prec (Q,\vv)$ if $|\vv'|<|\vv|$, or $|\vv'|=|\vv|$ and $|Q'|\ge 3>|Q|$.
\end{center}
Here $|\mathbf v|=\sum_{i\in Q_0}\mathbf v_i$. One can verify that $\prec$ is indeed a partial order. Note that if $(Q'',\vv'')\prec (Q',\vv')\prec (Q,\vv)$, then $|\vv''|<|\vv|$, therefore every chain $(Q,\vv)\succ (Q',\vv')\succ\cdots$ terminates after finite steps. We prove the theorem by induction on this partial order. Note that $|Q_0|=1$ case has been proven in the subsection \ref{subsec: case of one vertex}, so we assume that $|Q_0|\ge 2$ and the theorem is true for all $(Q',\vv')$ with flat moment map such that $(Q',\vv')\prec (Q,\vv)$.\\

\noindent $\bullet$ If $\theta$ is generic and $\vv$ is indivisible, then Remark \ref{rmk: smoothness} shows that $\mathcal{M}^{\theta}_{\lambda}(Q,\vv)$ is smooth. \\

\noindent $\bullet$ If $\theta$ is generic and $\vv$ is divisible, i.e. $\vv=n\ww, n\in \mathbb N_{>1}, \ww\in \mathbb N^{Q_0}$, then for arbitrary $\lambda\in Z$ we take $x\in \mu^{-1}(\lambda)^{\theta\mathrm{-ss}}$ such that $G(\mathbf{v})\cdot x$ is closed in $\mu^{-1}(\lambda)^{\theta\mathrm{-ss}}$. Either $x$ is simple, or $x$ decomposes as $$x=x_1^{\oplus k_1}\oplus\cdots\oplus x_r^{\oplus k_r},$$ where $x_t$ are $\theta$-stable with dimension vectors $s_t\ww$. If $x$ is simple then $\mu$ is smooth at $x$. If $x$ is decomposable, then Lemma \ref{lem: transverse slice} shows that there is an isomorphism
\begin{align*}
    \mathcal{M}^{\theta}_{\lambda}(Q,\mathbf{v})^{\wedge}_x   \cong \mathcal{M}_{0}(\hat Q,\hat\vv)^{\wedge}_0
\end{align*}
for a new quiver $(\hat Q,\hat\vv)$. We claim that
$$|\hat\vv|<|\vv|.$$In effect, In effect, the equation $\vv=\sum_t k_ts_t\ww$ implies that $|\hat\vv|\le|\vv|$ and the equality holds if and only $|\ww|=1$, but if this is the case then $\vv$ is supported at a single vertex thus $|Q_0|=1$, this contradicts with the assumption that $|Q_0|\ge 2$. Therefore $(\hat Q,\hat v)\prec (Q,\vv)$ and by induction hypothesis the Theorem \ref{thm: reducedness} is true for generic $\theta$.\\

\noindent $\bullet$ If $|Q_0|\ge 3$, then we take $\lambda\neq 0$ and take $x\in \mu^{-1}(\lambda)^{\theta\mathrm{-ss}}$ such that $G(\mathbf{v})\cdot x$ is closed in $\mu^{-1}(\lambda)^{\theta\mathrm{-ss}}$. Either $x$ is simple, or $x$ decomposes as $$x=x_1^{\oplus k_1}\oplus\cdots\oplus x_r^{\oplus k_r},$$ where $x_t$ are $\theta$-stable with dimension vectors $\mathbf{v}^{(t)}$. If $x$ is simple then $\mu$ is smooth at $x$. If $x$ is decomposable, then Lemma \ref{lem: transverse slice} shows that there is an isomorphism
\begin{align*}
    \mathcal{M}^{\theta}_{\lambda}(Q,\mathbf{v})^{\wedge}_x   \cong \mathcal{M}_{0}(\hat Q,\hat\vv)^{\wedge}_0
\end{align*}
for a new quiver $(\hat Q,\hat\vv)$. We claim that
$$|\hat\vv|<|\vv|.$$
In effect, the equation $\vv=\sum_t k_t\vv^{(t)}$ implies that $|\hat\vv|\le|\vv|$ and the equality holds if and only $|\vv^{(t)}|=1$ for all $t$, and $|\vv^{(t)}|=1$ is equivalent to $\vv^{(t)}=\ee_{i_t}$ for some vertex $i_t\in Q_0$, therefore $|\hat\vv|=|\vv|$ only happens when $\lambda\cdot \ee_i=0$ for all $i\in Q_0$, but this forces $\lambda=0$, which contradicts with our choice of $\lambda$. Hence we have 
$$(\hat Q,\hat\vv)\prec (Q,\vv)$$
and we conclude that $\mathcal{M}^{\theta}_{\lambda}(Q,\mathbf{v})$ is normal, because for every point in $\mathcal{M}_{\lambda}(Q,\mathbf{v})$ there is a preimage in $\mu^{-1}(\lambda)^{\theta\mathrm{-ss}}$ with closed $\GL(\vv)$-orbit. By the flatness of $\mu$, $\mathcal{M}^{\theta}_{Z}(Q,\mathbf{v})|_{Z\backslash 0}$ is normal. Again by the flatness of $\mu$, $\mathcal{M}^{\theta}_{0}(Q,\mathbf{v})$ has codimension $|Q_0|-1\ge 2$ in $\mathcal{M}^{\theta}_{Z}(Q,\mathbf{v})$, and  $\mu^{-1}(0)^{\theta\mathrm{-ss}}$ has codimension $|Q_0|-1\ge 2$ in $\mu^{-1}(Z)^{\theta\mathrm{-ss}}$. Since $\mu^{-1}(Z)^{\theta\mathrm{-ss}}$ is Cohen-Macaulay, we can apply the Lemma \ref{Lemma: Normality} and conclude that $\mathcal{M}^{\theta}_{Z}(Q,\mathbf{v})$ is normal. Since Theorem \ref{thm: reducedness} is true for $\mathcal M^{\theta}_{\lambda}(Q,\vv)$ with generic $\theta$, we can apply Lemma \ref{lemma: normality of family implies reducedness} thus Theorem \ref{thm: reducedness} is true for all $\mathcal M^{\theta}_{\lambda}(Q,\vv)$.\\

\noindent $\bullet$ If $|Q_0|=2$, the only $(\theta,\lambda)$ that is not generic is $(0,0)$, so we only need to show that $\mathcal M_Z(Q,\vv)$ is normal, according to Lemma \ref{lemma: generic case} and \ref{lemma: normality of family implies reducedness}. In the view of Lemma \ref{Lemma: Normality}, we need to show that there exists a Zariski open subset $U\subset \mathcal M_0(Q,\vv)$ such that $U$ is reduced and $\dim\xi^{-1}(\mathcal M_0(Q,\vv)\backslash U)<\dim \mu^{-1}(0)$ where $\xi:\mu^{-1}(0)\to \mathcal M_0(Q,\vv)$ is the quotient map. Lemma \ref{lem: transverse slice} shows that for every $x\in \mathcal M_0(Q,\vv)$, the formal completion $\mathcal M_0(Q,\vv)^{\wedge}_x$ is isomorphic to $\mathcal M_0(\hat Q,\hat \vv)^{\wedge}_0$ for some $(\hat Q,\hat \vv)$. So it is enough to show that there exists a constructible subset $W\subset \mathcal M_0(Q,\vv)$ such that $\forall x\in \mathcal M_0(Q,\vv)\backslash W$ the quiver $(\hat Q,\hat \vv)$ associated to $x$ is $\prec (Q,\vv)$, and $\dim\xi^{-1}(W)<\dim \mu^{-1}(0)$.

The prospective choice of $W$ is $\bigcup_{\tau}\mathcal M_0(Q,\vv)_{\tau}$ where $\tau$ are the representation types
\begin{align}
    \tau=(k_1,\vv^{(1)};\cdots;k_r,\vv^{(r)}),
\end{align}
such that $\vv^{(t)}$ are $\ee_1$ or $\ee_2$ for all $t$. Precisely, we are going to show that one of the following situations must happen:
\begin{itemize}
    \item[(a)] $\pp(\vv)>\vv_1\pp(\ee_1)+\vv_2\pp(\ee_2)$.
    \item[(b)] For representation type $\tau=(k_1,\vv^{(1)};\cdots;k_r,\vv^{(r)})$ as above, and assume moreover
\begin{align*}
    \pp(\vv)=\sum_{t=1}^r\pp(\vv^{(t)}),
\end{align*}
then $r$ must be greater than $2$.
\end{itemize}
If (a) is true, then we take $W=\bigcup_{\tau}\mathcal M_0(Q,\vv)_{\tau}$, where $\tau=(k_1,\vv^{(1)};\cdots;k_r,\vv^{(r)})$, so $\dim\xi^{-1}(W)<\dim \mu^{-1}(0)$ according to \cite[Corollary 6.4]{crawley2001normality}, and $\forall x\in \mathcal M_0(Q,\vv)\backslash W$ the associated $(\hat Q,\hat \vv)$ to $x$ must have $|\hat\vv|<|\vv|$ (the argument is the same as the $|Q_0|\ge 3$ case). If (b) is true, then we take $W=\bigcup_{\tau}\mathcal M_0(Q,\vv)_{\tau}$, where $\tau=(k_1,\vv^{(1)};\cdots;k_r,\vv^{(r)})$ such that
\begin{align*}
    \pp(\vv)>\sum_{t=1}^r\pp(\vv^{(t)}).
\end{align*}
In this situation $\dim\xi^{-1}(W)<\dim \mu^{-1}(0)$ according to \cite[Corollary 6.4]{crawley2001normality}, and $\forall x\in \mathcal M_0(Q,\vv)\backslash W$ the associated $(\hat Q,\hat \vv)$ to $x$ must have $|\hat\vv|<|\vv|$ or $|\hat{Q}_0|\ge 3$. Both situations imply the normality of $\mathcal M_Z(Q,\vv)$ by Lemma \ref{Lemma: Normality}.

To show that one of (a) or (b) must happen, we proceed case-by-case. Let $(\vv_1,\vv_2)=(N,K)$ and we can assume that $N\ge 2, K\ge 2$ because the case of $\vv_1=1$ (or symmetrically $\vv_2=1$) has been proven in Corollary \ref{cor: thm in rank 1}. Since the quiver will be doubled when considering moment map, without loss of generality we can assume that all arrows between vertices $1$ and $2$ are pointing from $1$ to $2$, and the adjacency matrix is
\begin{align}
    \begin{bmatrix}
    a & b\\
    0 & c
    \end{bmatrix},
\end{align}
note that $b>0$ (in order that $Q$ is connected), and we assume that $a\le c$ (otherwise we just reverse the direction of arrows).\\
%we use the notation $Q_{abc}$ for such quiver.\\

\noindent\textbf{(1)} $c\ge a\ge 1$. We claim that (a) is true in this case. In effect, 
\begin{align*}
    \pp(\vv)-N\pp(\ee_1)-K\pp(\ee_2)&=(a-1)(N-1)N+(c-1)(K-1)K+bNK+1-N-K\\
    &\ge NK-N-K+1>0.
\end{align*}

\noindent\textbf{(2)} $a=0,c>1$. We claim that (a) is true in this case. Consider $\ww=(\ww_1,\ww_2)=(1,K)$, then it is easy to see that $\ww\in \Sigma_0$, thus $\pp(\ww)>\pp(\ee_1)+K\pp(\ee_2)$, therefore
\begin{align*}
    \pp(\vv)\ge (N-1)\pp(\ee_1)+\pp(\ww)>N\pp(\ee_1)+K\pp(\ee_2).
\end{align*}

\noindent\textbf{(3)} $a=0,b>1$. We claim that (a) is true in this case. Consider $\ww=(\ww_1,\ww_2)=(1,1)$, then it is easy to see that $\ww\in \Sigma_0$, thus $\pp(\ww)>\pp(\ee_1)+\pp(\ee_2)$, therefore
\begin{align*}
    \pp(\vv)\ge (N-1)\pp(\ee_1)+(K-1)\pp(\ee_2)+\pp(\ww)>N\pp(\ee_1)+K\pp(\ee_2).
\end{align*}

\noindent\textbf{(4)} $(a,b,c)=(0,1,0)$. We claim that this case can not happen. In fact
\begin{align*}
    \pp(\vv)=-N^2-K^2+NK+1=1-(N-K)^2-NK\le 1-NK<0,
\end{align*}
by our assumption that $N\ge 2,K\ge 2$. This contradicts with the flatness assumption.

\noindent\textbf{(5)} $(a,b,c)=(0,1,1)$ and $K> N+1$. We claim that (a) is true in this case. In fact
\begin{align*}
    \pp(\vv)-N\pp(\ee_1)-K\pp(\ee_2)=NK-N^2+1-K=(N-1)(K-N-1)>0.
\end{align*}

\noindent\textbf{(6)} $(a,b,c)=(0,1,1)$ and $K< N+1$. This case can not happen, because $$\pp(\vv)-N\pp(\ee_1)-K\pp(\ee_2)=(N-1)(K-N-1)<0,$$this ccontradicts with the flatness assumption.

\noindent\textbf{(7)} $(a,b,c)=(0,1,1)$ and $K= N+1\ge 4$. This case can not happen, because 
\begin{align*}
    \pp(\vv)-(N-2)\pp(\ee_1)-\pp(2\ee_1+K\ee_2)=NK-N^2+1-(2K-3)=3-K<0,
\end{align*}
this contradicts with the flatness assumption.

\noindent\textbf{(8)} $(a,b,c)=(0,1,1)$ and $N=2, K=3$. We claim that (b) is true in this case. By a straightforward computation we see that the only representation type $\tau=(k_1,\vv^{(1)};\cdots;k_r,\vv^{(r)})$ such that $\pp(\vv)=\sum_{t=1}^r\pp(\vv^{(t)})$ is
\begin{align*}
    \tau=(2,\ee_1;1,\ee_2;1,\ee_2;1,\ee_2),
\end{align*}
therefore $r=4$ and (b) holds.\\

All cases have been covered, and this finishes the proof of Theorem \ref{thm: reducedness}.
\end{proof}

\section{Type A Dynkin Quivers and Finite W-Algebras}\label{sec: W-algebra}

\subsection{Flatness of moment maps for affine type A Dynkin quivers with framings}\label{subsec: affine type A}

Let $(Q,\mathbf{v},\dd)$ be a framed quiver with framing vector $\dd$ (assume $\dd\neq 0$). Following Crawley-Boevey, we define the associated unframed quiver $(Q^{\dd},\mathbf{v}^{\dd})$ as $Q^{\dd}_0=Q_0\sqcup \{\infty\}$ (union of vertices of $Q$ with an extra vertex denoted by $\infty$), and arrows in $Q^{\dd}$ are those from $Q$ and for each vertex $i\in Q_0$ attach $\dd_i$-copies of arrows from $\infty$ to $i$, and set $\mathbf{v}^{\dd}_i=\mathbf{v}_i$ if $i\in Q_0$ and $\mathbf{v}^{\dd}_{\infty}=1$. From the construction we see that the group
\begin{align}
    G(\mathbf{v}^{\dd})=\prod_{i\in Q^{\dd}_0}\GL(\mathbf{v}^{\dd}_i)/
    \mathbb C^{\times}\cong \prod_{i\in Q_0}\GL(\mathbf{v}_i)=:\GL(\mathbf{v})
\end{align}
acts on $R(Q^{\dd},\mathbf{v}^{\dd})$. Note that the deformation space $Z$ can be identified with $\prod_{\vv_i\neq 0}\mathbb C$, the identification is by $\lambda\mapsto \lambda^{\dd}$ such that $\lambda^{\dd}_i=\lambda_i$ for $i\in Q_0$ and $\lambda^{\dd}_{\infty}=-\sum_{i\in Q_0}\lambda_i\vv_i$. Similarly the stability space $\Theta$ can be identified with $\prod_{\vv_i\neq 0}\mathbb Q$ via $\theta\mapsto \theta^{\dd}$ such that $\theta^{\dd}_i=\theta_i$ for $i\in Q_0$ and $\theta^{\dd}_{\infty}=-\sum_{i\in Q_0}\theta_i\vv_i$. We define the Nakajima quiver schemes associated to $(Q,\vv,\dd)$ as:
\begin{align}
    \mathcal M^{\theta}_Z(Q,\vv,\dd):=\mathcal M^{\theta^{\dd}}_Z(Q^{\dd},\vv^{\dd}),\; \mathcal M^{\theta}_{\lambda}(Q,\vv,\dd):=\mathcal M^{\theta^{\dd}}_{\lambda^{\dd}}(Q^{\dd},\vv^{\dd}).
\end{align}

\begin{proposition}\label{prop: flatness in type A}
If $Q$ is an affine type A Dynkin quiver, let $\vv$ and $\dd$ be dimension vector and framing vector (assume $\dd\neq 0$), then the moment map $\mu$ is flat if and only if the following condition is satisfied:
\begin{align}\label{eqn: flatness type A}
    \ee_I\cdot(\dd-C_Q\vv)\ge -1,
\end{align}
for arbitrary subset $I\subset Q_0$ such that nodes in $I$ are connected by arrows in $Q$, and $\ee_I=\sum_{i\in I}\ee_i$. 
\end{proposition}

\begin{proof}
By Lemma \ref{lemma: flatness}, the flatness of $\mu$ is equivalent to that for every decomposition $\vv^{\dd}=\vv'+\vv^{(1)}+\cdots+\vv^{(r)}$ with $\vv^{(1)},\cdots,\vv^{(r)}$ being roots of $Q$, the inequality $\pp(\vv^{\dd})\ge \pp(\vv')+\pp(\vv^{(1)})+\cdots +\pp(\vv^{(r)})$ holds.

Let $\uu=\vv^{\dd}-\vv'$, then it is easy to see that $ \pp(\vv^{(1)})+\cdots +\pp(\vv^{(r)})\le \min_{i\in Q_0} \uu_i$, and the equality can be achieved. On the other hand $\pp(\vv^{\dd})- \pp(\vv')=\uu\cdot (\dd-C_Q\vv)+\frac{1}{2}(\uu,\uu)_Q$, so the flatness of moment map is equivalent to that $\forall \uu\in \mathbb N^{Q_0} $ such that $ 0\le \uu\le \vv$, the inequality
\begin{align}\label{eqn: inequality type A}
    \uu\cdot (\dd-C_Q\vv)+\frac{1}{2}(\uu,\uu)_Q\ge \min_{i\in Q_0} \uu_i
\end{align}
holds. Let us discuss two situations separately: (1) $\Supp(\vv)\neq Q$; (2) $\Supp(\vv)= Q$.

In the case (1), the right hand side of \eqref{eqn: inequality type A} is zero and $Q$ can be regarded as a type A quiver by removing one node with zero dimension vector. Moreover \eqref{eqn: inequality type A} holds if and only if it holds on each connected component of $\Supp(\vv)$, so let us assume that $Q'=\Supp(\vv)$ is connected. Let $I$ be a subset of $Q'_0$ such that nodes in $I$ are connected by arrows in $Q'$, and set $\uu=\ee_I$, then \eqref{eqn: inequality type A} implies that $\ee_I\cdot(\dd-C_Q\vv)\ge -1$. On the other hand if $\ee_I\cdot(\dd-C_Q\vv)\ge -1$ holds for all $I\subset Q'_0$ such that nodes in $I$ are connected by arrows in $Q'$, then Lemma \ref{lem: decomposition of test vector} implies that $\forall \uu\in \mathbb N^{Q'_0} $ such that $ 0\le \uu\le \vv$ we have a decomposition $\uu=\sum_{\alpha=1}^n\ee_{I_{\alpha}}$, such that $(\ee_{I_{\alpha}},\ee_{I_\beta})_Q\ge 0$ for any pair $1\le \alpha,\beta\le n$, therefore
\begin{align*}
    \uu\cdot (\dd-C_Q\vv)+\frac{1}{2}(\uu,\uu)_Q\ge \sum_{\alpha=1}^n \left(\ee_{I_{\alpha}}\cdot (\dd-C_Q\vv)+\frac{1}{2}(\ee_{I_{\alpha}},\ee_{I_{\alpha}})_Q\right)\ge 0
\end{align*}
This shows that \eqref{eqn: inequality type A} is equivalent to \eqref{eqn: flatness type A} being true for all $I\subset Q_0$ such that $I$ is a contained in a connected component of $\Supp(\vv)$. Observe that if $\vv_i=0$ then $\ee_i\cdot(\dd-C_Q\vv)\ge 0$, it follows that if \eqref{eqn: flatness type A} is true for all $I\subset Q_0$ such that $I$ is a contained in a connected component of $\Supp(\vv)$, then it is true for all $I$. This completes the proof in the case (1).

In the case (2), by taking $\uu=\ee_I$ we see that \eqref{eqn: inequality type A} implies that \eqref{eqn: flatness type A} holds for all $I$. Note that the case $I=Q_0$ is automatically true, since we assume that $\dd\neq 0$. Conversely, assume that \eqref{eqn: flatness type A} holds for all $I$, then let us write $\uu=m\delta+\sum_{\alpha=1}^n\ee_{I_{\alpha}}$ ($\delta$ is the imaginary root of $Q$) such that $I_{\alpha}\neq Q_0$ for all $1\le \alpha\le n$, and $(\ee_{I_{\alpha}},\ee_{I_\beta})_Q\ge 0$ for any pair $1\le \alpha,\beta\le n$, the existence of such decomposition follows from Lemma \ref{lem: decomposition of test vector}. Then we have
\begin{align*}
     \uu\cdot (\dd-C_Q\vv)+\frac{1}{2}(\uu,\uu)_Q\ge m\delta\cdot\dd+\sum_{\alpha=1}^n \left(\ee_{I_{\alpha}}\cdot (\dd-C_Q\vv)+\frac{1}{2}(\ee_{I_{\alpha}},\ee_{I_{\alpha}})_Q\right)\ge m=\min_{i\in Q_0} \uu_i.
\end{align*}
This completes the proof in the case (2).
\end{proof}

\begin{lemma}\label{lem: decomposition of test vector}
If $Q$ is an affine type A Dynkin quiver, and $\uu\in \mathbb N^{Q_0}$, then there exists subsets $I_1,\cdots,I_n$ of $Q_0$ such that nodes in $I_{\alpha}$ are connected by arrows in $Q$ for all $\alpha=1,\cdots,n$, and $\uu=\sum_{\alpha=1}^n\ee_{I_{\alpha}}$, and $(\ee_{I_{\alpha}},\ee_{I_\beta})_Q\ge 0$ for any pair $1\le \alpha,\beta\le n$.
\end{lemma}

\begin{proof}
Without loss of generality, let us assume that $\Supp(\uu)$ is connected, otherwise we can restrict to connected components. Let us take $I_1$ to be the set of nodes in $\Supp(\uu)$, note that for any $J\subset I_1$, we have $(\ee_J,\ee_{I_1})_Q\ge 0$. By induction on $\sum_{i\in Q_0}\uu_i$, we have $I_2,\cdots,I_n$ such that $\uu-\ee_{I_1}=\sum_{\alpha=2}^n\ee_{I_{\alpha}}$ and $(\ee_{I_{\alpha}},\ee_{I_\beta})_Q\ge 0$ for any pair $2\le \alpha,\beta\le n$. Now $I_{\alpha}\subset I_1$ for all $\alpha$, thus it suffices to take $I_1,\cdots,I_n$.
\end{proof}

\subsection{Quantization of quiver schemes}\label{subsec: Quantization of quiver schemes}

Let $\mathcal A_{\hbar}$ be a flat $\mathbb C[\hbar]$ algebra such that $A:=\mathcal A_{\hbar}/(\hbar)$ is commutative. Suppose that $\mathbb C^{\times}$ acts on $\mathcal A_{\hbar}$ by automorphisms such that $\hbar$ has weight $2$. Let $\mathfrak{g}$ be a Lie algebra with an action on $\mathcal A_{\hbar}$, i.e. a Lie homomorphism $\mathfrak{g}\to \mathrm{Der}(\mathcal A_{\hbar})$. Suppose that $\Phi_{\hbar}:\mathfrak{g}\to \mathcal A_{\hbar}$ is a $\mathbb C[\hbar]$-linear map such that image of $\Phi_{\hbar}$ is in the weight $2$ subspace of $\mathcal A_{\hbar}$. We furthermore assume that $\Phi_{\hbar}$ is a Lie algebra homomorphism, where the Lie bracket $[-,-]_{\hbar}$ on $\mathcal A_{\hbar}$ is $[a,b]_{\hbar}=\frac{1}{\hbar}[a,b]$.

\begin{definition}
A Lie algebra homomorphism $\Phi_{\hbar}:\mathfrak{g}\to \mathcal{A}_{\hbar}$ as above is called a \textit{quantum moment map} if $\forall a\in \mathfrak{g},b\in \mathcal A_{\hbar}$,
\begin{align}
    [\Phi_{\hbar}(a),b]_{\hbar}=a\cdot b.
\end{align}
Here $a\cdot b$ means the action of $a$ on $b$.
\end{definition}
Let $\chi\in (\mathfrak{g}/[\mathfrak{g},\mathfrak{g}])^*$ and form the shift $\mathfrak{g}_{\hbar \chi}=\{a-\hbar\langle \chi,a\rangle\}\subset \mathfrak{g}\oplus \mathbb C\hbar$, then it is easy to see that $(\mathcal A_{\hbar}\Phi_{\hbar}(\mathfrak{g}_{\hbar \chi}))^{\mathfrak{g}}$ is a two-sided ideal of $\mathcal A_{\hbar}^{\mathfrak{g}}$. More generally, let $U$ be a linear space and $\varphi:U\to (\mathfrak{g}/[\mathfrak{g},\mathfrak{g}])^*$ be a linear map, and form the shift $\mathfrak{g}_U=\{a-\varphi^*(a)\}\subset \mathfrak{g}\oplus U^*$, then $(\mathcal A_{\hbar}[U^*]\Phi_{\hbar}(\mathfrak{g}_{U}))^{\mathfrak{g}}$ is a two-sided ideal of $\mathcal A_{\hbar}^{\mathfrak{g}}[U^*]$.

\begin{definition}
Define the quantum Hamiltonian reductions $\mathcal A_{\hbar}\sslash_{\hbar\chi} \mathfrak{g}=\mathcal A_{\hbar}^{\mathfrak{g}}/(\mathcal A_{\hbar}\Phi_{\hbar}(\mathfrak{g}_{\hbar \chi}))^{\mathfrak{g}}$ and $\mathcal A_{\hbar}\sslash_{U} \mathfrak{g}=\mathcal A_{\hbar}^{\mathfrak{g}}[U^*]/(\mathcal A_{\hbar}[U^*]\Phi_{\hbar}(\mathfrak{g}_{U}))^{\mathfrak{g}}$.
\end{definition}

Note that we also have classical Hamiltonian reduction $A\sslash_{U} \mathfrak{g}=A^{\mathfrak{g}}[U^*]/(A[U^*]\Phi(\mathfrak{g}_{U}))^{\mathfrak{g}}$, and it is easy to see that there is a natural algebra homomorphism $(\mathcal A_{\hbar}\sslash_{U} \mathfrak{g})/(\hbar)\to A\sslash_{U} \mathfrak{g}$.

\begin{lemma}\label{lem: quantization commutes with reduction}
Suppose that $\mathfrak{g}$ is reductive, and $\{e_1,\cdots,e_n\}$ is a basis of $\ker(\varphi^*)\subset\mathfrak{g}$ such that $\{\Phi(e_i)\}_{i=1}^n$ is a regular sequence in $A=\mathcal{A}_{\hbar}/(\hbar)$, then there is an isomorphism
\begin{align}
    (\mathcal A_{\hbar}\sslash_{U} \mathfrak{g})/(\hbar)\cong A\sslash_{U} \mathfrak{g}.
\end{align}
\end{lemma}
\begin{proof}
See \cite[Lemma 3.3.1]{losev2012isomorphisms}.
\end{proof}

Apply the above construction to the quiver representation of $(Q,\mathbf{v})$, and we set $\mathcal A_{\hbar}=W_{\hbar}(T^*R(Q,\vv))$ (Weyl algebra of the symplectic vector space $T^*R(Q,\vv)$) with $\mathbb C^{\times}$ acting on the vector space $T^*R(Q,\vv)$ of weight $-1$, $G=G(\vv)$, , and $\Phi_{\hbar}:\mathfrak{g}(\vv)\to W_{\hbar}(T^*R(Q,\vv))$ is the (unique) degree $2$ lift of the comoment map $\mu^*: \mathfrak{g}(\vv)\to \mathbb C[T^*R(Q,\vv)]$. For a linear map $U\to Z=(\mathfrak{g}(\vv)/[\mathfrak{g}(\vv),\mathfrak{g}(\vv)])^*$, call the quantum Hamiltonian reduction $W_{\hbar}(T^*R(Q,\vv))\sslash_U \mathfrak{g}(\vv)$ the \textit{quantized quiver scheme}, denoted by $\mathcal A_{\hbar,U}(Q,\vv)$. Since $\mathfrak{g}(\vv)$ is reductive, the natural homomorphism $\mathcal A_{\hbar,U}(Q,\vv)/(\hbar)\to \mathbb C[\mathcal{M}_U(Q,\vv)]$ is surjective, where $\mathcal{M}_U(Q,\vv)=\mathcal{M}_Z(Q,\vv)\times_Z U$. A special case of Lemma \ref{lem: quantization commutes with reduction} reads:
\begin{lemma}\label{lem: quantization commutes with reduction_quiver case}
If the moment map $\mu: T^*R(Q,\vv)\to \mathfrak{g}(\vv)^*$ is flat, then the epimorphism $\mathcal A_{\hbar,U}(Q,\vv)/(\hbar)\twoheadrightarrow \mathbb C[\mathcal{M}_U(Q,\vv)]$ is an isomorphism, and $\mathcal A_{\hbar,U}(Q,\vv)$ is flat over $\mathbb C[U][\hbar]$.
\end{lemma}

\begin{proof}
If $\mu$ is flat then any basis $\{e_1,\cdots,e_n\}$ of $\ker (\mathfrak{g}(\vv)\to U^*)$ is mapped to a regular sequence in $\mathbb C[T^*R(Q,\vv)]$, so by \ref{lem: quantization commutes with reduction} $\mathcal A_{\hbar,U}(Q,\vv)/(\hbar)\to \mathbb C[\mathcal{M}_U(Q,\vv)]$ is an isomorphism. Moreover the associated graded of $\mathcal A_{\hbar,U}(Q,\vv)$ with respect to the $\hbar$-filtration is $ \mathbb C[\mathcal{M}_U(Q,\vv)][\hbar]$, which is flat over $\mathbb C[U][\hbar]$, therefore $\mathcal A_{\hbar,U}(Q,\vv)$ is flat over $\mathbb C[U][\hbar]$.
\end{proof}

We can sheafify the construction of quantum Hamiltonian reduction, at the cost of taking $\hbar$-completion. We will not repeat the definition here, see \cite[3.4]{losev2012isomorphisms} for detail. A particular case that we will focus on is the following: $G(\vv)$ acts on $\widehat{W}_{\hbar}(T^*R(Q,\vv))$ with quantum moment map $\Phi_{\hbar}$, here $\widehat{W}_{\hbar}(T^*R(Q,\vv))$ is the completion of the Weyl algebra of symplectic vector space $T^*R(Q,\vv)$ in the $\hbar$-adic topology, and it can be sheafified on $T^*R(Q,\vv)$, assume that $\vv$ is \textit{indivisible} and $\theta\in \Theta$ is a \textit{generic} stability condition, then for every linear map $U\to Z=(\mathfrak{g}(\vv)/[\mathfrak{g}(\vv),\mathfrak{g}(\vv)])^*$ we have a sheaf of associative flat $\mathbb C[\![\hbar]\!]$-algebras $\widehat{W}_{\hbar}(T^*R(Q,\vv))\sslash^{\theta}_U G(\vv)$ on the smooth scheme $\mathcal M^{\theta}_U(Q,\vv)$ (see Remark \ref{rmk: smoothness}), denoted by $\mathcal O_{\hbar,\mathcal M^{\theta}_U(Q,\vv)}$. Note that $\mathcal O_{\hbar,\mathcal M^{\theta}_U(Q,\vv)}/(\hbar)\cong \mathcal O_{\mathcal M^{\theta}_U(Q,\vv)}$, i.e. $\mathcal O_{\hbar,\mathcal M^{\theta}_U(Q,\vv)}$ is a quantization of $\mathcal O_{\mathcal M^{\theta}_U(Q,\vv)}$.

\begin{proposition}\label{prop: isomorphism different stability condition}
Suppose that the quiver $(Q,\vv)$ is such that $\vv$ is indivisible and the moment map $\mu: T^*R(Q,\vv)\to \mathfrak{g}(\vv)^*$ is flat, then the natural homomorphism
\begin{align}
    W_{\hbar}(T^*R(Q,\vv))^{G(\vv)}\to \Gamma(\mathcal M^{\theta}_U(Q,\vv),\mathcal O_{\hbar,\mathcal M^{\theta}_U(Q,\vv)})
\end{align}
gives rise to an isomorphism between $\mathcal A_{\hbar,U}(Q,\vv)$ and the sub-algebra of $\mathbb C^{\times}$-finite elements in $\Gamma(\mathcal M^{\theta}_U(Q,\vv),\mathcal O_{\hbar,\mathcal M^{\theta}_U(Q,\vv)})$.
\end{proposition}

\begin{proof}
We need to check the conditions (i), (ii), and (iii) in \cite[Lemma 4.2.4]{losev2012isomorphisms}. Condition (i) is checked by Lemma \ref{lem: quantization commutes with reduction_quiver case}; condition (ii) holds because $\vv$ is indivisible and $\theta$ is generic; condition (iii) follows from that $p^{\theta}:\mathcal M^{\theta}_{\lambda}(Q,\vv)\to \mathcal M_{\lambda}(Q,\vv)$ is birational and Poisson for all $\lambda\in Z$ (see Corollary \ref{cor: resolution of singularity}), and the Poisson structure on $\mathcal M^{\theta}_{\lambda}(Q,\vv)$ is symplectic.
\end{proof}

\subsection{Finite W-algebras}\label{subsec: Finite W-algebras}
Let $G$ be a reductive algebraic group, $\mathfrak{g}$ be the Lie algebra of $G$. Pick a nilpotent element $e\in \mathfrak{g}$ and choose an $\mathfrak{sl}_2$-triple $e,f,h$. Set $\chi\in \mathfrak{g}^*$ be $(e,-)$. Consider the grading $\mathfrak{g}=\bigoplus_i \mathfrak{g}(i)$ by the eigenvalues of $\mathrm{ad}(h)$. It is easy to see that the skew-symmetric form $\langle\xi,\eta\rangle =\chi([\xi,\eta])$ is non-degenerate on $\mathfrak{g}(-1)$. Pick a Lagrangian $l\subset \mathfrak{g}(-1)$ and set $\mathfrak{m}:=l\oplus\bigoplus _{i\le -2}\mathfrak{g}(-i)$. Note that $\chi\in (\mathfrak{m}/[\mathfrak{m},\mathfrak{m}])^*$. Since $\mathfrak{m}$ is nilpotent, it exponentiates to an algebraic subgroup $M\subset G$. The finite W-algebra $\mathcal W_{\hbar}$ is defined as the quantum Hamiltonian reduction $U_{\hbar}(\mathfrak{g})\sslash_{\chi}M$, where $U_{\hbar}(\mathfrak{g})=T(\mathfrak{g})/(xy-yx-\hbar [x,y] \: |\: x,y\in \mathfrak{g})$.

Let $P\subset G$ be a parabolic subgroup and $P_0=[P,P]$, and let $\mathfrak{a}=\mathfrak{p}/\mathfrak{p}_0$. Consider the $\mathbb C[\mathfrak{a}^*][\hbar]$-algbera $\mathcal A_{\hbar}$ defined as $\mathbb C^{\times}$-finite elements in $\Gamma(G/P,D_{\hbar}(G/P_0)^{P/P_0})$, where $D_{\hbar}(G/P_0)$ is the sheaf of $\hbar$-adic differential operators and the action of $P/P_0$ on $D_{\hbar}(G/P_0)$ is induced from the right action on $G/P_0$, and the $\mathbb C^{\times}$-action on $D_{\hbar}(G/P_0)$ is inherited from $\mathbb C^{\times}$-action on $T^*(G/P_0)$ by scaling cotangent fibers of weight $-2$. Note that the $\mathbb C[\mathfrak{a}^*]$-algebra structure comes from a modification of the moment map $\Phi_{\hbar}:\mathfrak{a}\subset U_{\hbar}(\mathfrak{g})\to \mathcal{A}_{\hbar}$, which is defined as $a\mapsto \Phi_{\hbar}(a)-\langle \rho,a\rangle\hbar$, where $\rho$ is the half sum of positive roots of $\mathfrak{g}$. The parabolic finite W-algebra $\mathcal W^P_{\hbar,\mathfrak{a}}$ is defined as the quantum Hamiltonian reduction $\mathcal A_{\hbar}\sslash_{\chi} M$.

We focus on the type A case. Explicitly, $G=\SL_N$, and fix $r_1,\cdots,r_n\in \mathbb Z_{\ge 0}$ with $\sum_{i=1}^nr_i=N$, then $r_1,\cdots,r_n$ defines a parabolic subgroup $P\subset \SL_N$ as the stabilizer of a partial flag $\mathcal F=(0\subset F_1\subset F_2\subset\cdots\subset F_n=\mathbb C^N)$ with $\dim F_j=\sum_{i=1}^jr_i$. Also pick $\dd=(d_1,\cdots,d_{n-1})\in \mathbb Z^{n-1}_{\ge 0}$ with $\sum_{i=1}^{n-1}id_i=N$ and let $e\in \mathfrak{sl}_N$ be a nilpotent element whose Jordan type is $(1^{d_1},2^{d_2},\cdots,(n-1)^{d_{n-1}})$. Define $\vv=(v_1,\cdots,v_{n-1})$ by 
\begin{align}
    v_{n-1}=r_n,\; v_i=\sum_{j=i+1}^nr_j-\sum_{j=i+1}^n(j-i)d_j.
\end{align}
Below we assume that all $v_i$'s are non-negative. Let $Q$ be an $A_{n-1}$ quiver so we can identify $\vv$ as a dimension vector and $\dd$ as a framing vector. View $\mathfrak{a}=\mathfrak{p}/\mathfrak{p}_0$ as the space $\{\mathrm{diag}(x_1,\cdots,x_1,\cdots,x_n,\cdots,x_n)\}$ matrices such that $x_i$ appears $r_i$ times with $\sum_{i=1}^nr_ix_i=0$. Map $\mathfrak{a}$ to $\mathbb C^{Q_0}$ by sending $\mathrm{diag}(x_1,\cdots,x_n)$ to $\sum_{i=1}^{n-1}(\sum_{j=1}^ir_jx_j)\epsilon_i$. The composition of this map with the natural projection $\mathbb C^{Q_0}\twoheadrightarrow Z^*=\mathfrak{gl}(\vv)/[\mathfrak{gl}(\vv),\mathfrak{gl}(\vv)]$ defined by
\begin{align*}
    \epsilon_i\mapsto \begin{cases}
        \frac{1}{v_i}\mathrm{Id}_{v_i} &, v_i\neq 0\\ 
        0 &, v_i=0 \\
        \end{cases}
\end{align*}
gives rise to a map $\mathfrak{a}\to Z^*$. Note that this is an isomorphism if all $r_i$'s and $v_i$'s are positive. Define the $\mathbb C[Z][\hbar]$-algebra $$\mathcal W_{\hbar, Z}^P:=\mathcal W_{\hbar, \mathfrak{a}}^P\otimes _{\mathbb C[\mathfrak{a}^*]} \mathbb C[Z],$$where $\mathcal W_{\hbar, \mathfrak{a}}^P$ is the parabolic finite W-algebra associated to the aforementioned parabolic subgroup $P$ and nilpotent element $e$. Apply the construction in the previous subsection and we obtain a $\mathbb C[Z][\hbar]$-algebra $$\mathcal A_{\hbar,Z}(Q,\vv,\dd):=\mathcal A_{\hbar,Z}(Q^{\dd},\vv^{\dd}).$$Losev showed in \cite[Theorem 5.3.3]{losev2012isomorphisms} that if all $v_i$'s are positive and $r_i\ge r_{i+1}$ then $\mathcal A_{\hbar,Z}(Q,\vv,\dd)$ is $\mathbb C[Z][\hbar]$-linearly isomorphic to $\mathcal W_{\hbar, Z}^P$\footnote{If all $v_i$'s are positive and $r_i\ge r_{i+1}$, then all $r_i$'s are positive since $r_n=v_{n-1}>0$, and it follows that the map $\mathfrak{a}\to Z^*$ is an isomorphism.}. Here we state a refinement of \textit{loc. cit.}.
\begin{proposition}\label{prop: quantum quiver scheme and W-algebra}
There exists a $\mathbb C[Z][\hbar]$-linear epimorphism $\mathcal A_{\hbar,Z}(Q,\vv,\dd)\twoheadrightarrow \mathcal W_{\hbar, Z}^P$ of graded associative algebras. Moreover it is an isomorphism if and only if $r_i-r_j\ge -1$ for all $1\le i<j\le n$.
\end{proposition}

\begin{proof}
Recall that $\mathcal W_{\hbar, \mathfrak{a}}^P$ is the algebra of $\mathbb C^{\times}$-finite elements in $\Gamma(\widetilde{S}(e,P),\mathcal O_{\hbar,\widetilde{S}(e,P)})$ for a $\mathbb C^{\times}$-equivariant even quantization $\mathcal O_{\hbar,\widetilde{S}(e,P)}$ on the deformed Slodowy variety $\widetilde{S}(e,P)$ (see \cite[Lemma 5.2.1(2)]{losev2012isomorphisms}). Note that the Slodowy variety $S(e,P)$ is $\mathbb C^{\times}$-equivariantly symplectomorphic to $\mathcal M^{\theta}_0(Q,\vv,\dd)$ \cite{maffei2005quiver}, where $\theta=(-1,\cdots,-1)$ is a generic stability condition. The aforementioned quantized structure sheaf of $\mathcal M^{\theta}_Z(Q,\vv,\dd)$, denoted by $\mathcal O_{\hbar,\mathcal M^{\theta}_Z(Q,\vv,\dd)}$, is also an even quantization, moreover the quantum period maps $Z\to H^2_{\mathrm{DR}}(\mathcal M^{\theta}_0(Q,\vv,\dd))\cong H^2_{\mathrm{DR}}(S(e,P))$ and $\mathfrak{a}^*\to H^2_{\mathrm{DR}}(S(e,P))$ are intertwined by the map $Z\to \mathfrak{a}^*$ \cite[Lemma 4.6.5]{losev2012isomorphisms}, henceforth there is a $\mathbb C^{\times}$-equivariant isomorphism between formal schemes $\widetilde{S}(e,P)\times_{\mathfrak{a}^*} \widehat{Z}\cong \widehat{\mathcal M}^{\theta}_Z(Q,\vv,\dd)$ together with a $\mathbb C^{\times}$-equivariant isomorphism between sheaf of flat $\mathbb C[\![\hbar]\!]$-algebras $\mathcal O_{\hbar,\widetilde{S}(e,P)}\widehat{\otimes}_{\mathbb C[\![\mathfrak{a}^*]\!]} \mathbb C[\![Z]\!]\cong \mathcal{O}_{\hbar,\widehat{\mathcal M}^{\theta}_Z(Q,\vv,\dd)}$, where $\widehat{Z}$ is the completion of $Z$ at $0$ and $\widehat{\mathcal M}^{\theta}_Z(Q,\vv,\dd)$ is the completion of $\mathcal M^{\theta}_Z(Q,\vv,\dd)$ along $\mathcal M^{\theta}_0(Q,\vv,\dd)$. Thus there exists a $\mathbb C[Z][\hbar]$-linear isomorphism between $\mathcal W_{\hbar, Z}^P$ and the algebra of $\mathbb C^{\times}$-finite elements in $\Gamma(\widehat{\mathcal M}^{\theta}_Z(Q,\vv,\dd),\mathcal O_{\hbar,\widehat{\mathcal M}^{\theta}_Z(Q,\vv,\dd)})$. 

By \cite[Theorem 12]{maffei2005quiver}, the projection $p^{\theta}:\mathcal M^{\theta}_0(Q,\vv,\dd)\to \mathcal M_0(Q,\vv,\dd)$ maps $\mathcal M^{\theta}_0(Q,\vv,\dd)$ birationally to its image in $\mathcal M_0(Q,\vv,\dd)$ and its image is a normal variety, thus $R^ip^{\theta}_*\mathcal O_{\mathcal M^{\theta}_0(Q,\vv,\dd)}=0$ for all $i>0$ and the natural map $\mathbb C[\mathcal M_0(Q,\vv,\dd)]\to \Gamma({\mathcal M}^{\theta}_0(Q,\vv,\dd),\mathcal O_{\hbar,{\mathcal M}^{\theta}_0(Q,\vv,\dd)})$ is surjective. Since the $\mathbb C^{\times}$-action on $\mathcal M^{\theta}_Z(Q,\vv,\dd)$ is positive, we conclude that $R^ip^{\theta}_*\mathcal O_{\mathcal M^{\theta}_Z(Q,\vv,\dd)}=0$ for all $i>0$ and the natural map $\mathbb C[\mathcal M_Z(Q,\vv,\dd)]\to \Gamma({\mathcal M}^{\theta}_Z(Q,\vv,\dd),\mathcal O_{\hbar,{\mathcal M}^{\theta}_Z(Q,\vv,\dd)})$ is surjective. 

Following verbatim argument as the proof of \cite[Lemma 4.2.4]{losev2012isomorphisms}, there is a $\mathbb C[Z][\![\hbar]\!]$-linear map of graded associative algebras $\mathcal A_{\hbar,Z}(Q,\vv,\dd)^{\wedge_{\hbar}}\to \Gamma({\mathcal M}^{\theta}_Z(Q,\vv,\dd),\mathcal O_{\hbar,{\mathcal M}^{\theta}_Z(Q,\vv,\dd)})$ and it fits into a commutative diagram
\begin{center}
\begin{tikzcd}
\mathcal A_{\hbar,Z}(Q,\vv,\dd)^{\wedge_{\hbar}} \arrow[d]  \arrow[r] & \Gamma({\mathcal M}^{\theta}_Z(Q,\vv,\dd),\mathcal O_{\hbar,{\mathcal M}^{\theta}_Z(Q,\vv,\dd)}) \arrow[d]\\
\mathbb C[\mathcal M_Z(Q,\vv,\dd)]  \arrow[r] & \Gamma({\mathcal M}^{\theta}_Z(Q,\vv,\dd),\mathcal O_{{\mathcal M}^{\theta}_Z(Q,\vv,\dd)})
\end{tikzcd}
\end{center}
where $\mathcal A_{\hbar,Z}(Q,\vv,\dd)^{\wedge_{\hbar}}$ is the completion of $\mathcal A_{\hbar,Z}(Q,\vv,\dd)$ in the $\hbar$-adic topology. The bottom horizontal arrow and vertical arrows are epimorphisms, thus the top horizontal arrow is also an epimorphism since both domain and codomain are flat over $\mathbb C[\![\hbar]\!]$ and it is surjective modulo $\hbar$. Since $\mathcal A_{\hbar,Z}(Q,\vv,\dd)$ is the sub-algebra of $\mathbb C^{\times}$-finite elements in $\mathcal A_{\hbar,Z}(Q,\vv,\dd)^{\wedge_{\hbar}}$, it follows that $\mathcal A_{\hbar,Z}(Q,\vv,\dd)$ maps surjectively to the sub-algebra of $\mathbb C^{\times}$-finite elements in $\Gamma({\mathcal M}^{\theta}_Z(Q,\vv,\dd),\mathcal O_{\hbar,{\mathcal M}^{\theta}_Z(Q,\vv,\dd)})$. Since $\mathbb C^{\times}$-finite elements in $\Gamma(\widehat{\mathcal M}^{\theta}_Z(Q,\vv,\dd),\mathcal O_{\hbar,\widehat{\mathcal M}^{\theta}_Z(Q,\vv,\dd)})$ are the same as $\mathbb C^{\times}$-finite elements in $\Gamma({\mathcal M}^{\theta}_Z(Q,\vv,\dd),\mathcal O_{\hbar,{\mathcal M}^{\theta}_Z(Q,\vv,\dd)})$, the first assertion is proven.

Next, suppose that the epimorphism $\mathcal A_{\hbar,Z}(Q,\vv,\dd)\twoheadrightarrow \mathcal W_{\hbar, Z}^P$ we constructed above is an isomorphism, then it is an isomorphism modulo $Z^*$ and $\hbar$. Therefore, the composition $$\mathcal A_{\hbar,Z}(Q,\vv,\dd)/(Z^*,\hbar)\twoheadrightarrow \mathbb C[\mathcal M_0(Q,\vv,\dd)]\to \Gamma({\mathcal M}^{\theta}_0(Q,\vv,\dd),\mathcal O_{{\mathcal M}^{\theta}_0(Q,\vv,\dd)})$$is an isomorphism. This implies that $\mathbb C[\mathcal M_0(Q,\vv,\dd)]\to \Gamma({\mathcal M}^{\theta}_0(Q,\vv,\dd),\mathcal O_{{\mathcal M}^{\theta}_0(Q,\vv,\dd)})$ is injective, in particular $\dim \mathcal M_0(Q,\vv,\dd)=\dim\mathcal M^{\theta}_0(Q,\vv,\dd)$, then by Lemma \ref{lem: characterization of flatness} the moment map for $(Q,\vv,\dd)$ is flat. By Proposition \ref{prop: flatness in type A}, the flatness of moment map is equivalent to $\ee_I\cdot(\dd-C_Q\vv)\ge -1$ for arbitrary subset $I\subset Q_0$ such that nodes in $I$ are connected by arrows in $Q$. For $1\le i<j\le n$, let $\ee_{i,j}=\sum_{k=i}^{j-1}\ee_k$, then it is easy to see that $\ee_{i,j}\cdot(\dd-C_Q\vv)=r_i-r_j$, thus the flatness of moment map is equivalent to $r_i-r_j\ge -1$ for all $1\le i<j\le n$.

Conversely, assume that $r_i-r_j\ge -1$ for all $1\le i<j\le n$, then the moment map is flat, thus by Proposition \ref{prop: isomorphism different stability condition} $\mathcal A_{\hbar,Z}(Q,\vv,\dd)$ maps isomorphically to $\mathbb C^{\times}$-finite elements in $\Gamma({\mathcal M}^{\theta}_Z(Q,\vv,\dd),\mathcal O_{\hbar,{\mathcal M}^{\theta}_Z(Q,\vv,\dd)})$, which is isomorphic to $\mathcal W_{\hbar, Z}^P$ by our previous discussion. 
\end{proof}

\bibliographystyle{unsrt}
\bibliography{Bib}

\end{document}